\documentclass[a4paper]{amsart}
\usepackage{graphicx}
\usepackage{amsmath,amsthm,amsfonts,amssymb,longtable,bm,booktabs,tikz-cd,colortbl}
\usepackage[hang,small,bf]{caption}
\usepackage{subcaption}
\usepackage[margin=30truemm]{geometry}
\usepackage[T1]{fontenc}
\theoremstyle{plain}
\newtheorem{thm}{Theorem}[section]
\newtheorem{prp}[thm]{Proposition}
\newtheorem{lem}[thm]{Lemma}
\newtheorem{cor}[thm]{Corollary}
\theoremstyle{definition}
\newtheorem{dfn}[thm]{Definition}

\newtheorem{rmk}[thm]{Remark}

\newcommand{\Aut}{\operatorname{Aut}}
\newcommand{\kernel}{\operatorname{Ker}}
\newcommand{\image}{\operatorname{Im}}

\newcommand{\id}{\operatorname{id}}
\newcommand{\GL}{\operatorname{GL}}
\newcommand{\PGL}{\operatorname{PGL}}
\newcommand{\diag}{\operatorname{diag}}
\newcommand{\Pic}{\operatorname{Pic}}
\newcommand{\maxx}{\operatorname{max}}

\newcommand{\PSU}{\operatorname{PSU}}

\title{The octanomial normal forms of cubic surfaces with applications to automorphisms}
\author{China Kaneko}
\address{Department of Mathematics, 
Faculty of Science and Technology, Tokyo University of Science}
\email{china\_kaneko@alumni.tus.ac.jp}

\begin{document}
 
 \begin{abstract}
 	We will show that 
	in any characteristic
	every nonsingular cubic surface is projectively isomorphic to the surface 
	given by the octanomial normal form. 
	This normal form is discovered 
	in \cite{Pan20} only in characteristic $0$ in a heavily computational way.  
	Our proof is more conceptual and characteristic-free. 
	As an application, 
	we give octanomial normal forms of the strata of the coarse moduli space of cubic surfaces defined in \cite{DD19},
	which preserve most specialization with respect to automorphisms. 
 \end{abstract}
 
 \maketitle
 
 \setcounter{tocdepth}{1}
 \tableofcontents
 
\section{Introduction}

The aim of this paper is 
to discuss some good properties of the new normal form of nonsingular cubic surfaces in the $3$-dimensional projective space $\mathbb{P}^3$, discovered by M. Panizzut, E-C. Sert\"oz and B. Sturmfels \cite{Pan20}. 

After the discovery of the $27$ lines by A. Cayley and G. Salmon in $1849$, 
nonsingular cubic surfaces in $\mathbb{P}^3$ have been well-researched. 
The configuration of the $27$ lines on a nonsingular cubic surface is independent of the surface. 
Classical theory introduced essential tools defined as suitable configuration of some lines, 
such as the Schl\"{a}fli's double-six and the triad of tritangent trios, 
which work in characteristic-free. 
Recently,
K. Ranestad and B. Sturmfels proposed the $27$ questions about tropicalization of nonsingular cubic surfaces \cite{RKS19}, 
and nonsingular cubic surfaces come into the spotlight again. 
In particular, the research in aspects of tropical geometry and computational method is rapidly developing. 

In the history of nonsingular cubic surfaces, 
the discovery of good normal forms has played an important role. 
A \textit{normal form} of nonsingular cubic surfaces is a homogeneous cubic form in $4$ variables which represents general nonsingular cubic surfaces. 
For example, the Sylvester normal form is a classical normal form in characteristic $\neq 2,3$. 
Over $\mathbb{C}$, 
the elementary symmetric polynomials of degree $1,\dots, 5$ of the parameters of this normal form  
provide the coordinates of the birational model $\mathbb{P}(1,2,3,4,5)$ of the coarse moduli space of nonsingular cubic surface  \cite[Subsection 9.4.5]{Dol12}. 

The Sylvester normal form is a useful normal form 
but there are some nonsingular cubic surfaces which are defined by no Sylvester forms even in characteristic $0$. 
In addition, it is even not a normal form in characteristic $2,3$. 
The Emch normal form is one in any characteristic 
(\cite{Emc31} in characteristic $0$, \cite[Theorem 6.3]{DD19} in any characteristic) 
although it does not represent all nonsingular cubic surfaces. 
 The Cremona's hexahedral form is one which represents all, 
although it is not a normal form in characteristic $3$. 

In 2020, M. Panizzut and the others presented the following normal form \cite{Pan20}: 
\[ 
 a_0x_1x_2x_3 + a_1x_0x_2x_3 + a_2x_0x_1x_3 + a_3x_0x_1x_2 
	 + a_4x_0^2x_1 + a_5x_0x_1^2 + a_6x_2^2x_3 + a_7x_2x_3^2 = 0, 
\] 
where $a_0,\dots ,a_7$ were parameters. 
They called this equation the \textit{octanomial form}. 
They showed that in characteristic $0$
every cubic surface given by blowing up $6$ points in $\mathbb{P}^2$ was projectively isomorphic to the surface given by the above equation. 
They discovered this normal form to solve a question about the tropicalization of a cubic surface. 
Their methods to introduce and to investigate this normal form are heavily computational. 

As the first main theorem of this paper, 
we show that the octanomial normal form represents \textit{all} nonsingular cubic surfaces in \textit{any} characteristic. 
Our method works in any characteristic, 
and is more conceptual and computation-free. 

 \begin{thm}\label{octa}
	In any characteristic, 
	every nonsingular cubic surface is projectively isomorphic to a surface in $\mathbb{P}^3$ given by the equations: 
	\begin{equation}\label{eq:octa}
	x_0x_1(x_0 + x_1 + a_3x_2 + a_2x_3) + x_2x_3(a_1x_0 + a_0x_1 + x_2+x_3) = 0, 
	\end{equation}
	where $a_0,\dots ,a_3$ are parameters. 
	We call the equation (\ref{eq:octa}) the \textit{octanomial normal form}, 
	and call $a_0,\dots ,a_3$ the \textit{octanomial parameter}. 
\end{thm}

It follows from Theorem \ref{octa} that 
there is the surjection $\Phi\colon \mathbb{A}_k^4 \to \mathcal{M}(k)$, 
where $\mathbb{A}_k^4$ is the octanomial parameter space 
and $\mathcal{M}(k)$ is the coarse moduli space of nonsingular cubic surfaces 
over an algebraically closed field $k$ of characteristic $p\geq 0$. 
We expect that the octanomial parameters give 
a substitute of a chart of $\mathcal{M}(k)$ in a characteristic-free way. 
We show that for a nonsingular cubic surface $X$ 
the number of the octanomial parameters is 
$25920/|\Aut (X)|$ in $p\neq 3$ and $8640/|\Aut (X)|$ in $p=3$. 
In particular, 
the degree of $\Phi$ is $25920$ in $p\neq 3$ and $8640$ in $p=3$. 
See Remark \ref{octa_para} for details.

We apply Theorem \ref{octa} to automorphisms. 
For every cubic surface $X$, 
there is the injective group homomorphism $\Phi_X \colon \Aut (X) \to W(\mathsf{E}_6)$ 
from the automorphism group of $X$ to the Weyl group of the root lattice $\mathsf{E}_6$, 
unique up to conjugacy (Corollary \ref{auto}). 
By using this injective homomorphism 
and the classification of the conjugacy classes of elements in $W(\mathsf{E}_6)$ by R-W. Carter \cite{Car72} 
(denoted by the Atlas labeling $1A,\dots ,12C$), 
I. Dolgachev and A. Duncan \cite{DD19} defined 
the stratification in $\mathcal{M}(k)$ in the following way. 
Let $C$ be a conjugacy class of $W(\mathsf{E}_6)$ 
and $g$ be an automorphism of a cubic surface $X$. 
If $\Phi_X(g)$ is in $C$, $g$ is called \textit{of type $C$}. 
The \textit{stratum $C$} is the subset of $\mathcal{M}(k)$ 
of the isomorphism classes of cubic surfaces which admit some automorphisms of type $C$. 
For example, 
the stratum $2A$ consists of the isomorphism classes of cubic surfaces 
which are preserved by exchanging $x_0$ and $x_1$ in a suitable coordinates. 
Each stratum is a closed subvariety of $\mathcal{M}(k)$, which is unirational \cite[Corollary 14.2]{DD19}. 
A stratum $C'$ is a \textit{specialization} of a stratum $C$ 
if $C$ strictly contains $C'$, denoted by $C \to C'$. 
For example,
for an automorphism $g$ of type $12A$,  $g^4$ is one of type $3A$, 
so there is the specialization $3A \to 12A$.
The specialization of the strata in characteristic $0$ is shown in Figure \ref{fig:stra1}. 
See \cite[Appendix. Figures 2,3,4]{DD19} for other characteristics. 
Each row corresponds to the strata of dimension one less than the preceding row, 
where the stratum $1A$ has dimension $4$ and the strata $5A, 3C, 12A$ and $8A$ have dimension $0$. 
I. Dolgachev and A. Duncan classified full automorphism groups of nonsingular cubic surfaces 
by using the stratification in any characteristic. 

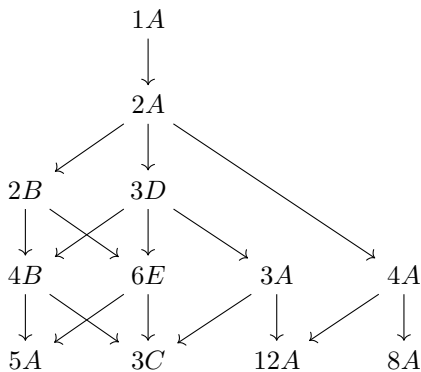
\begin{figure}[h]
	\centering
		\begin{tikzcd}
		 & 
		 1A  
		 \arrow[d]
		 & &\\
		 & 
		 2A 
		 \arrow[d]
		 \arrow[ld]
		 \arrow[rrdd]
		 & &  \\
		2B 
		 \arrow[d]
		 \arrow[rd]
		& 
		3D
		\arrow[d]
		\arrow[ld]
		\arrow[rd]
		 & & \\
		 4B
		 \arrow[d]
		 \arrow[rd] 
		 & 
		 6E
	 	\arrow[d]
		 \arrow[ld]
		  & 
		  3A
		  \arrow[d]
	  	\arrow[ld] 
		&
		  4A
		  \arrow[d]
	  	\arrow[ld] 
	  	\\
	 	5A & 3C & 12A & 8A  	
		\end{tikzcd}
		\caption{The specialization of the strata in $p\neq 2,3,5$}
		\label{fig:stra1}
\end{figure}

As our second main theorem, 
we present normal forms of the strata in $\mathcal{M}(k)$, 
which are octanomial forms 
and preserve most of the specializations of the strata. 

\begin{thm}\label{main thm}
	Let $\mathcal{M}(k)$ be the coarse moduli space of nonsingular cubic surfaces 
	over an algebraically closed field $k$ of characteristic $p\geq 0$. 
	Each stratum $C$ in $\mathcal{M}(k)$ has the octanomial normal form 
	\[ 
	x_0x_1(x_0 + x_1 + a_3x_2 + a_2x_3) + x_2x_3(a_1x_0 + a_0x_1 + x_2 + x_3) = 0
	\]
	with the parameters $a_0,\dots ,a_3$ given by the third column of Table \ref{tab:main_thm}.
	We call this the $C$-octanomial normal form.  
\end{thm}

\begin{table}[h]	
	\centering
	\caption{The octanomial normal forms and the automorphisms of the strata} \label{tab:main_thm} 
	\begin{tabular}{cccc}
		\toprule
		Stratum& 
		Char& 
		Octanomial parameter& 
		Reference \\
		\midrule
		$2A$  &
		any &
		$a_0=a_1$ &
		Thm \ref{2A-octa} \\
		\midrule
		$2B$ & 
		any &
		$a_0=a_1, \ a_2=a_3$ & 
		Thm \ref{2B-octa} \\
		\midrule
		$3A$ & 
		$\neq 3$ &
		$a_0=a_1=0, \ a_2+ \zeta_3 a_3=0$ &
		 Thm \ref{3A-octa} \\
		 &
		 $= 3$ &
		 $a_0=a_1=0, \ a_2 + a_3 = 0$ & 
		\\
		\midrule
		 $3C$ & 
		 $\neq 3$ & 
		$a_0=a_1=a_2=a_3=0$ & 
		Thm \ref{3C-octa} \\
		\midrule
		$3D$  & 
		any &
		$a_0=a_1=0$ &  
		Thm \ref{3D-octa} \\
		\midrule
		$4A$ & 
		$\neq 2$ &  
		\begin{tabular}{l}
		$a_0 = a_1$, \\
		\small{$q^3 + 4pq + 8 -2a_0(3q^2 -4a_0q + 4p) = 0$, }\\
		\small{$3q -p^2 + a_0 (2a_0^2q - a_0(q^2 + 2p) + pq + 2) = 0$,} \\
		\footnotesize{where $p=a_2 + a_3$, $q=a_2a_3$. }\\
		\end{tabular}
		& 
		Thm \ref{4A-octa} \\
		&
		$=2$ & 
		\multicolumn{1}{|>{\columncolor[gray]{0.7}}c|}
		{(Same as $6E$)}&
		Thm \ref{6E=4B=4A-octa_char2} \\
		\midrule
		$4B$ &		
		$\neq 2$ & 
		$a_0=a_1=a_2=a_3$ &  
		Thm \ref{4B-octa} \\
		& 
		$=2$ &
		\multicolumn{1}{|>{\columncolor[gray]{0.7}}c|}{(Same as $6E$)}& 
		Thm \ref{6E=4B=4A-octa_char2} \\
		\midrule
		$5A$  & 
		$\neq 2, 5$ &
		$a_0=a_1=0 , \ a_2=a_3=-2$ & 
		Thm \ref{5A-octa} \\
		& 
		$=2$ &
		\multicolumn{1}{|>{\columncolor[gray]{0.7}}c|}{(Same as $3C$)} &
		Thm \ref{3C=5A=12A-octa_char2} \\
		\midrule
		$6E$ & 
		any &
		$a_0 = a_1 = 0, \ a_2=a_3$ &  
		Thm \ref{6E-octa} \\
		\midrule
		$8A$ &		
		$\neq 2$ & 
		\begin{tabular}{l}
		\small{Satisfying the condition of $4A$}, \\
		\small{and $\gamma (1-a_0a_2) - 2 \delta (1-a_0a_3) = 0$,} \\
		\footnotesize{
		where $\alpha = a_2^2 / (4(a_0a_2-1))$,
		$\beta = a_3^2 / (4(a_0a_3-1))$,} \\ 
		\footnotesize{
		$A = a_3\alpha + a_2\beta$, } \\
		\footnotesize{
		$\gamma = \beta^2(1-a_0a_2) + 2 \alpha \beta (1 - a_0a_3) + a_0 \beta - a_3(1 - A)/2$, }\\
		\footnotesize{
		$\delta = \alpha^2 (1-a_0a_3) + 2 \alpha \beta (1 - a_0a_2) + a_0 \alpha - a_2(1 - A)/2$.}
		\end{tabular}
		& 
		Thm \ref{8A-octa} \\
		& 
		$=3$ & 
		$a_0 = a_1 = 0, \ a_2 = i, \ a_3 = -i$ & 
		Thm \ref{8A=12A-octa_char3} \\
		\midrule
		12A & 
		$\neq 2,3$ & 
		\begin{tabular}{l}
		$a_0=a_1=0$, \\
		$a_2+\zeta_3a_3=0$, \\
		and $a_3^6 + 4\zeta_3(1-\zeta_3)a_3^3 -8 = 0$\\
		\end{tabular} & 
		Thm \ref{12A-octa} \\
		& 
		$=2$ & 
		\multicolumn{1}{|>{\columncolor[gray]{0.7}}c|}{(Same as $3C$)} & 
		Thm \ref{3C=5A=12A-octa_char2} \\
		& 
		$=3$ &
		\multicolumn{1}{|>{\columncolor[gray]{0.7}}c|}{(Same as $8A$)} & 
		Thm \ref{8A=12A-octa_char3}
		 \\
		 \bottomrule
	\end{tabular}
\end{table}

The $4A$, $8A$ and $12A$-octanomial normal forms in $p\neq 2$ are more complicated than others 
because enough geometric properties associated with these strata have not been found.

For almost any stratum $C$, 
\textit{every} nonsingular cubic surface admitting an automorphism of type $C$ is 
isomorphic to the surface defined by the $C$-octanomial normal form. 
Exceptionally, 
for one of the strata $3A$, $4A$ in $p\neq 2$, or $4B$ in $p\neq 2$, 
\textit{not every} nonsingular cubic surface is.

These octanomial normal forms of the strata preserve most specialization of the strata. 
For example, 
the $3D$-octanomial normal form ($a_0=a_1=0$) is given 
as a special case of the $2A$-octanomial normal form ($a_0 = a_1$). 
Every specialization is preserved in $p=2$ 
although some are not in $p\neq 2$, 
such as $3D \to 4B$ and $4B \to 5A$. 
For details, see Remark \ref{3D_to_4B}, \ref{4B_to_5A}.

\subsection*{Structure of the paper}
This paper is structured as follows. 
In Section \ref{sec:configu}, 
we fix notation, 
and recall the basics of nonsingular cubic surfaces 
such as the $\mathsf{E}_6$-lattice and the configuration of the $27$ lines. 
In particular, we show the well-known result: 
the permutation group of the configuration of the $27$ lines of a nonsingular cubic surface 
is isomorphic to the Weyl group of the root lattice $\mathsf{E}_6$. 
In Section \ref{sec:det}, 
we introduce 
a determinantal representation and a Cayley--Salmon equation, 
which depends on the coordinates of $\mathbb{P}^3$. 
We prove Theorem \ref{octa} in Section \ref{sec:octa}, 
by using the classical fact that every nonsingular cubic surface has Cayley--Salmon equations 
(we recall it as Lemma \ref{CS:tri}). 
In Section \ref{sec:main}, 
we show Theorem \ref{main thm} 
by recalling the arrangement of the Eckardt points 
or making coordinate change for each non-empty stratum.

\section{The configuration of lines on a cubic surface}\label{sec:configu}
Throughout this paper, 
we work over an algebraically closed field $k$ of characteristic $p \geq 0$. 
A \textit{cubic surface} is a nonsingular cubic surface in $\mathbb{P}^3_k$. 

By blowing up six points in $\mathbb{P}^2$ which are in general position
 (i.e. no three points on a line and the six points not on a conic), 
we get a nonsingular cubic surface in $\mathbb{P}^3$. 
Conversely, every nonsingular cubic surface is given in such a way. 
In particular, 
every cubic surface is embedded in $\mathbb{P}^3$ by the anti-canonical embedding,  
so we see that the automorphism group of a cubic surface is embedded in $\PGL_4(k)$.

Let $X$ be a cubic surface given by blowing up the six points $P_1,\dots ,P_6 \in \mathbb{P}^2$ in general position, 
and $\pi\colon X \to \mathbb{P}^2$ the projection. 
The Picard group $\Pic (X)$ of $X$ is a free abelian group of rank $7$ 
which is a lattice with respect to the intersection pairing. 
It is generated by the linear equivalence classes of the pullback of a line in $\mathbb{P}^2$ (denoted by $[L]$) 
and the exceptional lines above $P_1,\dots ,P_6$ (denoted by $[E_1],\dots ,[E_6]$). 
In Subsection \ref{subsec:lattice}, 
we will see fundamental properties of  
the lattice $I^{1,6}$ which is isometric to the Picard group of a cubic surface.  
The cubic surface $X$ contains exactly $27$ lines, given by the following \cite[V. Theorem 4.9]{Har77}: 
\begin{itemize}
	\item the exceptional curves $E_i$, $i=1,\dots ,6$ \ ($6$ of these); 
	\item the strict transform $F_{ij}$ of the line in $\mathbb{P}^2$ containing $P_i$ and $P_j$, 
	$1\leq i <j \leq 6$ \ ($15$ of these); 
	\item the strict transform $G_j$ of the conic in $\mathbb{P}^2$ containing the five $P_i$ for $i\neq j$, 
	$j=1,\dots ,6$ \ ($6$ of these). 
\end{itemize}
Each of them has self-intersection $-1$, 
and they are the only irreducible curves with negative self-intersection on $X$. 
In particular, 
the $27$ lines of any cubic surface has the following configuration: 
$E_i$ does not meet $E_j$ for $i\neq j$; 
$E_i$ meets $F_{jk}$ if and only if $i\in \{j,k\}$; 
$E_i$ meets $G_j$ if and only if $i\neq j$; 
$F_{ij}$ meets $F_{kl}$ if and only if $i,j,k,l$ are all distinct; 
$F_{jk}$ meets $G_i$ if and only if $i\in \{j,k\}$; 
$G_i$ does not meet $G_j$ for $i\neq j$. 
In Subsection \ref{subsec:lines}, 
we will introduce some data which are defined as the configuration of the lines on a cubic surface satisfying some conditions, 
which are useful to treat cubic surfaces in characteristic-free.

\subsection{The $\mathsf{E}_6$-lattice}\label{subsec:lattice}
A \textit{lattice} is a free abelian group $M$ of finite rank $r$  
equipped with a symmetric bilinear form $( \ , \ )\colon M \times M \to \mathbb{Z}$. 
For $\bm{v} \in M$, we write $(\bm{v}^2)$ for $(\bm{v}, \bm{v})$. 
A \textit{sublattice} of the lattice $M$ is a subgroup $N$ with the restriction of the bilinear form of $M$.
A \textit{morphism} of lattices is a morphism of abelian groups which preserves the bilinear forms. 
An invertible morphism of lattices is called an \textit{isometry}. 
The group of isometries of a lattice $M$ to itself is called the \textit{orthogonal group} of $M$, denoted by $O(M)$. 

Let $I^{1,6}$ be the lattice
$\mathbb{Z}^{7}$ equipped with the symmetric bilinear form 
defined by the diagonal matrix $\diag(1,-1,\dots, -1)$ 
with respect to the standard basis
$\bm{e_0} = (1,0,\dots ,0)$, 
$\bm{e_1} = (0,1,0,\dots ,0)$, 
$\dots$, 
$\bm{e_6}= (0,\dots ,0,1)$.
We define the vector $\bm{k} := -3 \bm{e_0} + \sum_{i=1}^6 \bm{e_i}$ in $I^{1,6}$ 
and the sublattice $\mathsf{E}_6 := \{\bm{v} \in I^{1,6}; \ (\bm{v}, \bm{k}) = 0 \}$ in $I^{1,6}$. 

A \textit{root} in $\mathsf{E}_6$ is a vector $\bm{\alpha} \in \mathsf{E}_6$ with $(\bm{\alpha} ^2) = -2$. 
Any root in $\mathsf{E}_6$ is one of the following 
$\pm \bm{\alpha_{\maxx}}, \pm \bm{\alpha_{ij}}$ or $\pm \bm{\alpha_{ijk}}$, 
and there are exactly $72$ roots in $\mathsf{E}_6$ \cite[Lemma 8.2.7]{Dol12}: 
\begin{itemize}
	\item
	$\bm{\alpha_{\maxx}} := 2\bm{e_0} - \bm{e_1} - \cdots -\bm{e_6}$ \ ($1$ of these);
	\item
	$\bm{\alpha_{ij}} := \bm{e_i} - \bm{e_j}$ ,$1\leq i < j \leq 6$ \ ($15$ of these); 
	\item
	$\bm{\alpha_{ijk}} := \bm{e_0} - \bm{e_i} -\bm{e_j} - \bm{e_k}$, $1\leq i <j < k \leq 6$ \ ($20$ of these). 
\end{itemize}
For any root $\bm{\alpha} \in \mathsf{E}_6$, 
we define the \textit{reflection} in $\bm{\alpha}$ by 
\[ r_{\bm{\alpha}} \colon I^{1,6} \longrightarrow I^{1,6} ; \ \bm{v} \longmapsto \bm{v} + (\bm{v}, \bm{\alpha})\bm{\alpha}. \]
We can see that $r_{\bm{\alpha}}$ is an isometry of order $2$
which preserves the vector $\bm{k}$, 
hence $r_{\bm{\alpha}} \in O(\mathsf{E}_6)$. 
The \textit{Weyl group} of $\mathsf{E}_6$, denoted by $W(\mathsf{E}_6)$, 
is the subgroup of the orthogonal group $O(\mathsf{E}_6)$ generated by the reflections in all roots in $\mathsf{E}_6$. 

A vector $\bm{v} \in I^{1,6}$ is called an \textit{exceptional vector} 
if $(\bm{v}, \bm{k}) = (\bm{v}^2) = -1$.
Any exceptional vector in $\mathsf{E}_6$ is one of the following,  
and there are exactly $27$ exceptional vectors \cite[Proposition 8.2.19]{Dol12}: 
\begin{itemize}
	\item
	$\bm{e_i}$, $i=1,\dots ,6$ \ ($6$ of these); 
	\item 
	$\bm{e_0} - \bm{e_i} - \bm{e_j}$, $1 \leq i < j \leq 6$ \ ($15$ of these); 
	\item
	$2 \bm{e_0} - \sum_{i=1}^6 \bm{e_i} + \bm{e_j}$, $j=1,\dots ,6$ \ ($6$ of these). 
\end{itemize}
Since reflections preserve intersection numbers, 
the Weyl group $W(\mathsf{E}_6)$ acts on the set of exceptional vectors in $I^{1,6}$. 
When $g \in O(I^{1,6})$ fixes each exceptional vector, 
we have $g(\bm{e_i}) =\bm{e_i}$ for $i=1,\dots ,6$ and 
$g(\bm{e_0}) - g(\bm{e_1}) - g(\bm{e_2}) = g(\bm{e_0} - \bm{e_1} - \bm{e_2}) = \bm{e_0}-\bm{e_1} - \bm{e_2}$, 
so $g$ fixes $\bm{e_i}$ for $i=0,\dots ,6$. 
Hence $W(\mathsf{E}_6)$ acts on the set of exceptional vectors faithfully.

\subsection{Lines on a cubic surface}\label{subsec:lines}
 
A \textit{marking} of a cubic surface $X$ is an isometry $\varphi\colon I^{1,6} \to \Pic (X)$ with $\varphi (\bm{k}) = [K_X]$. 
Given a marking $\varphi$ of $X$, 
the (ordered) $6$ lines $\varphi(\bm{e}_1), \dots , \varphi(\bm{e}_6)$ are mutually skew, 
so there is a labeling with these $6$ lines are $E_1,\dots ,E_6$ 
because the labeling of the other $21$ lines are uniquely determined by the condition of the intersection pairing with the above $6$ lines.  
Conversely, given $6$ mutually skew lines $\ell_1,\dots ,\ell_6$ on $X$, 
there is a unique marking $\varphi$ of $X$ with $\ell_i = \varphi(\bm{e}_i)$ for all $i$ \cite[V.Proposition 4.10]{Har77}. 

We fix a marking $\varphi\colon I^{1,6} \to \Pic (X)$ of a cubic surface $X$, 
then the labeling of the $27$ lines with $\varphi(\bm{e_i}) = [E_i]$ ($i=1,\dots ,6$) is fixed. 
We will define some concepts which come from the configuration of the $27$ lines.

\begin{dfn} 
	A \textit{sixer} is a set of six mutually skew lines. 
	There are exactly $72$ sixers on $X$
	(given by the first column of Table \ref{sixer:root:tw_cub_curve}). 
	A \textit{double-six} is a pair of sixeres 
	$(\{\ell_1, \dots, \ell_6 \}, 
	\{\ell'_1,\dots , \ell'_6\})$
	for which the conditions $(\ell_i. \ell'_i) = 0$ and $(\ell_i. \ell'_j) = 1$ for $i\neq j$ hold under a suitable indexing. 
	This is denoted by 
	\[ \begin{pmatrix}
		\ell_1 & \ell_2 & \ell_3 & \ell_4 & \ell_5 & \ell_6 \\ 
		\ell'_1 & \ell'_2 & \ell'_3 & \ell'_4 & \ell'_5 & \ell'_6
	\end{pmatrix}.\]
	We note that any permutation of rows and columns of the above matrix defines the same double-six. 
	There are exactly $36$ double-sixes 
	(given by the pair of sixeres $D$ and $-D$ of the first column of Table \ref{sixer:root:tw_cub_curve}). 
	By picking $3$ columns up from a double-six, 
	we get a configuration of $6$ lines, 
	which is called a \textit{half of double-six}. 
	It follows from Table \ref{sixer:root:tw_cub_curve} that 
	any half of double-six extends  to a double-six uniquely.  
\end{dfn}

\begin{lem}[{\cite[Lemma 9.1.2]{Dol12}}]\label{Dol12:9.1.2}
	Let $D=\{\ell_1, \dots, \ell_6\}$ be a sixer. 
	Then there exists a unique root $\bm{\alpha}$ in $\mathsf{E}_6$ such that 
	$(\ell_i. \varphi(\bm{\alpha})) = 1$ for $i=1,\dots ,6$. 
	We call $\bm{\alpha}$ the root associated to $D$. 
	Moreover, 
	let $\ell'_i:=\varphi(r_{\bm{\alpha}}(\varphi^{-1}(\ell_i)))$ for each $i$, 
	then $\{\ell'_1, \dots , \ell'_6\}$ is the sixer 
	to which the root associated is $-\bm{\alpha}$.
	The pair 
	$(\{\ell_1, \dots , \ell_6\}, \{\ell'_1,\dots , \ell'_6\})$
	is a double-six. 
	There is a one-to-one correspondence between sixeres and roots in $\mathsf{E}_6$ 
	(given by the first and second columns of Table \ref{sixer:root:tw_cub_curve}). 
\end{lem}

\begin{table}[b]
	\centering
	\caption{The correspondence between sixeres, roots in $\mathsf{E}_6$, 
	and linear equivalence classes of the twisted cubic curves on a cubic surface.}\label{sixer:root:tw_cub_curve}
	\begin{tabular}{lllc}
	\toprule
	\small{Sixer} & \small{Root} & \small{Class of the twisted cubic curve} & \small{Number}  \\ 
	\midrule
	$D_{\maxx} = 
		\{E_1, E_2, E_3, E_4, E_5, E_6\}$ & 
	$\bm{\alpha_{\maxx}}$ & 
	$5\bm{e_0} - \sum_{t=1}^6 2\bm{e_t}$ &
	$1$ \\
	$-D_{\maxx} = 
		\{G_1, G_2, G_3, G_4, G_5, G_6\}$ &
	$-\bm{\alpha_{\maxx}}$ &
	$\bm{e_0}$ &
	$1$ \\ 
	$D_{ij} = 
		\{E_i, G_i, F_{jk}, F_{jl}, F_{jm}, F_{jn}\}$ &
	$\bm{\alpha_{ij}}$ & 
	$3\bm{e_0} - \sum_{t=1}^6 \bm{e_t} + \bm{e_i} - \bm{e_j}$ &
	$15$ \\
	$-D_{ij} = 
		\{E_j, G_j, F_{ik}, F_{il}, F_{im}, F_{in}\}$ &
	$\bm{-\alpha_{ij}}$ &
	$3\bm{e_0} - \sum_{t=1}^6 \bm{e_t} - \bm{e_i} + \bm{e_j}$ &
	$15$ \\ 
	$D_{ijk} =
		\{E_i, E_j, E_k, F_{mn}, F_{ln}, F_{lm}\}$ & 
	$\bm{\alpha_{ijk}}$ &
	$4\bm{e_0} - (\bm{e_l} + \bm{e_m} + \bm{e_n}) -2(\bm{e_i} + \bm{e_j} + \bm{e_k})$ &
	$20$ \\
	$-D_{ijk} = 
		\{F_{jk}, F_{ik}, F_{ij}, G_l, G_m, G_n\}$ &
	$-\bm{\alpha_{ijk}}$ &
	$2\bm{e_0} - (\bm{e_l} + \bm{e_m} + \bm{e_n})$ & 
	$20$ \\
	\bottomrule
	\end{tabular}
\end{table}

A \textit{twisted cubic curve} on $X$ is a curve on $X$ of degree $3$ and genus $0$ in $\mathbb{P}^3$. 
There are exactly $72$ linear equivalence classes of twisted cubic curves on $X$ \cite[Proposition 4.4]{Buc07}, 
and they also correspond to roots in $\mathsf{E}_6$. 

\begin{lem}\label{tw_cub_curve:root}
	There is a one-to-one correspondence between 
	linear equivalence class of twisted cubic curves on $X$ and roots in $\mathsf{E_6}$, 
	given by $\bm{c} \mapsto \bm{c} + \bm{k}$ for a linear equivalence class of a twisted cubic curve $\bm{c}$ 
	(See the second and third columns of Table \ref{sixer:root:tw_cub_curve}) . 
\end{lem}

\begin{prp}
The Weyl group $W(\mathsf{E}_6)$ is isomorphic to the permutation group of the configuration of the $27$ lines on $X$. 
\end{prp} 
\begin{proof}
	Let $G$ be the permutation group of the configuration of the $27$ lines on a cubic surface.  
	Namely, $G$ is the group of permutations of the $27$ lines which preserve the intersection numbers. 
	With respect to a given marking $\varphi$ of $X$, 
	each line on $X$ correspond to  an exceptional vector. 
	Because we have seen already that $W(\mathsf{E}_6)$ acts faithfully on the set of exceptional vectors, 
	we have the injective morphism $\psi\colon W(\mathsf{E}_6) \to G$. 
	
	We will show that $\psi$ is surjective. 
	A permutation of the configuration of the $27$ lines is determined by 
	the images of the $6$ lines $E_1, \dots , E_6$ which form the sixer $D_{\maxx}$
	since the images of the other $21$ lines are determined 
	by the intersection numbers with the images of $E_1,\dots ,E_6$. 
	So it is enough to show that for each sixer $D$,  
	\begin{enumerate} 
		\item[(i)]
		 there is an element $r \in W(\mathsf{E}_6)$ which induces a permutation with $r(D_{\maxx}) = D$; and 
		 \item[(ii)]
		 there are elements of $W(\mathsf{E}_6)$ which induce permutations of the $6$ lines of $D$. 
	\end{enumerate}
	Let $\bm{\alpha}$ be the root associated to $D$, 
	then $r_{\alpha}(D) = -D$ by Lemma \ref{Dol12:9.1.2}. 
	Hence it is sufficient to find the permutations $r_{ijk}, r_{ij} \in W(\mathsf{E}_6)$ such that 
	$r_{ijk}(D_{\maxx}) = D_{ijk}$ and  $r_{ij}(D_{\maxx}) = D_{ij}$ for each $i,j,k$. 
	By calculating directly, we have 
	$r_{\bm{\alpha_{lmn}}} (D_{\maxx}) = D_{ijk}$ and 
	$r_{\bm{\alpha_{jkl}}} \circ r_{\bm{\alpha_{jmn}}} (D_{\maxx}) = D_{ij}$, 
	where $\{i,j,k,l,m,n\} = \{1,\dots ,6\}$. This shows (i). 
	For (ii), 
	it is sufficient to show the statement for the sixer $D_{\maxx}$ thanks to (i). 
	The reflection $r_{\bm{\alpha_{ij}}}$ induces the permutation swapping $E_i$ and $E_j$, 
	and these transpositions generate all permutations of the $6$ lines in $D_{\maxx}$. 
\end{proof}

\begin{cor}\label{auto}
	There exists the injective homomorphism $\Phi_X \colon \Aut (X) \to W(\mathsf{E}_6)$, 
	unique up to conjugacy. 
\end{cor}

\begin{proof}
	Giving a marking $\varphi\colon I^{1,6} \to \Pic (X)$, 
	we have the labeling of the $27$ lines of $X$. 
 	The automorphism group of $X$ acts faithfully on the configuration of the $27$ lines, 
	there exists the injective group homomorphism $\Phi_X \colon \Aut (X) \to W(\mathsf{E}_6)$. 
	 For another marking $\varphi' \colon I^{1,6} \to \Pic (X)$ of $X$, 
	we get the injective homomorphism $\Phi'_X  \colon \Aut (X) \to W(\mathsf{E}_6)$ in the same way. 
	There is a permutation $r \in W(\mathsf{E}_6)$ such that $\Phi_X (g) = r^{-1} \Phi'_X (g) r$ for $g \in \Aut (X)$. 
	(Here, $r$ sends $\varphi (\bm{e}_i)$ to $\varphi' (\bm{e}_i)$ for $i=1,\dots ,6$.)
\end{proof}

\begin{dfn}
	A \textit{tritangent plane} of $X$ is a plane $\pi$ in $\mathbb{P}^3$ which intersects $X$ with $3$ distinct lines. 
	The set of the three coplanar lines is called a \textit{tritangent trio}. 
	There are exactly $45$ tritangent planes and tritangent trios, 
	given by the following: 
	\begin{itemize}
		\item $\pi_{ij} = \{E_i, G_j, F_{ij}\}$, $1\leq i\neq j \leq 6$ \ ($2 \times \binom{6}{2} = 30$ of these);
		\item $\pi_{ij,kl,mn} =  \{F_{ij}, F_{kl}, F_{mn}\}$, $\{ i,j,k,l,m,n\} = \{1,\dots,6\}$ \ 
		($\binom{6}{2}\times\binom{4}{2}\times1/3! = 15$ of these).
	\end{itemize} 
	
	A \textit{triad of tritangent trios} is a set of $3$ tritangent trios $\{\pi_1,\pi_2,\pi_3\}$ 
	in which any two of them have no common line. 
	For a triad of tritangent trios, 
	there uniquely exists another triad of tritangent trios which consists of the same $9$ lines. 
	We refer to the pair of triads of tritangent trios which consist of common $9$ lines  
	as a \textit{conjugate pair} of triads of tritangent trios. 
	We denote a conjugate pair of triads of tritangent trios 
	$(\{\pi_1,\pi_2,\pi_3\}, \{\pi'_1, \pi'_2, \pi'_3\})$ 
	by following, like a matrix: 
	\[
	 \begin{array}{c|ccc}
	 	& \pi_1' & \pi_2' & \pi_3' \\ \hline
		\pi_1 & \ell_{11} & \ell_{12} & \ell_{13} \\
		\pi_2 & \ell_{21} & \ell_{22} & \ell_{23} \\
		\pi_3 & \ell_{31} & \ell_{32} & \ell_{33} 
	\end{array}. 
	\]
	Here, 
	the set of the three lines $\ell_{i1},\ell_{i2},\ell_{i3}$ of the $i$-th row is the tritangent trio on the tritangent plane $\pi_i$, 
	and the set of three lines $\ell_{1j},\ell_{2j},\ell_{3j}$ of the $j$-th column is the tritangent trio on the tritangent plane $\pi'_j$. 
	We sometimes write the $9$ lines only, omitting the tritangent planes. 
	We note a conjugate pair of triads of tritangent trios is invariant 
	under permutation of rows and columns, and transposition of the matrix.   
	There are exactly $120$ conjugate pairs of triads of tritangent trios given by follwing: 
	\[ \begin{matrix}
		E_i & G_j & F_{ij} \\
		G_k & F_{jk} & E_j \\
		F_{ik} & E_k & G_i
	\end{matrix} 
	\hspace{15mm}
	\begin{matrix}
		E_i & G_j & F_{ij}\\
		G_k & E_l & F_{lk} \\
		F_{ik} & F_{jl} & F_{mn}
	\end{matrix} \hspace{15mm}
	\begin{matrix}
		F_{ij} & F_{lm} & F_{kn} \\
		F_{ln} & F_{ik} & F_{jm} \\
		F_{km} & F_{jn} & F_{il}		
	\end{matrix}.\]
	There are 
	$\binom{6}{3}=20$,
	 $\binom{6}{4} \times \binom{4}{2} = 90$,
	 $\binom{5}{3} = 10$ (by putting $i=1, m<n$) of these, respectively.
\end{dfn}

\section{Determinantal representations}\label{sec:det}

We will introduce a determinantal representation and a Cayley--Salmon equation of a cubic surface. 
Let $X \subset \mathbb{P}^3$ be a cubic surface, defined by a cubic form $f \in k[x_0,\dots ,x_3]$. 
A \textit{determinantal representation} of $X$ is a $3\times 3$ matrix $M$ of linear forms satisfying 
$\det M = \lambda f$ for some $\lambda \neq 0$. 
Two determinantal representations $M, M'$ are \textit{equivalent}, denoted by $M \sim M'$, 
if there are matrices $A,B \in \GL_{3}(k)$ such  that $M'= AMB$. 
The following theorem is classical. 
\begin{thm}[{\cite[Theorem 1.1]{Buc07}}]
	A smooth cubic surface $X$ in $\mathbb{P}^3$ allows exactly $72$ nonequivalent determinantal representations. 
	There is a one-to-one correspondence between: 
	\begin{enumerate}
		\item[(i)] 
		equivalence classes of determinantal representations of $X$; 
		\item[(ii)] 
		linear systems of twisted cubic curves on $X$; 
		\item[(iii)] 
		sixeres of $X$. 
	\end{enumerate}
\end{thm}
We survey the correspondence between (i) and (ii) of the above, according to \cite[\S 2]{Buc07}. 
Let $t_0,t_1,t_2$ be a homogeneous coordinates of $\mathbb{P}^2$. 
Given a sixer of $X$, 
by blowing down them, 
we have $6$ points $P_1,\dots ,P_6$ in $\mathbb{P}^2$ in a general position. 
By Hilbert--Burch theorem, 
there is a $3\times 4$ matrix $L$ of linear forms of $t_0,t_1,t_2$ 
whose minors form a basis of the linear system 
$| \mathcal{O}_{\mathbb{P}^2} (3) - \sum_{i=1}^6  P_{i} |$ 
of plane cubic curves with the assigned points. 
Let $M$ be a $3\times 3$ matrix of linear forms of $x_0,\dots ,x_3$ satisfying 
\begin{equation}\label{eq:det}
M 
\begin{bmatrix}
	t_0 \\ t_1 \\ t_2 
\end{bmatrix}
= 
L 
\begin{bmatrix}
	x_0 \\ x_1 \\ x_2 \\ x_3
\end{bmatrix}, 
\end{equation}
then $M$ is a determinantal representation. 
Conversely, 
let $M$ be a determinantal representation of $X$, 
and $L$ the matrix satisfying (\ref{eq:det}). 
The rank of $L$ in $P = {}^t \! (a_0,a_1,a_2)$ is $2$ if and only if 
the $3$ planes in $\mathbb{P}^3$ determined by $M \, {}^t \! (a_0,a_1,a_2) = 0$ intersects in a line. 
There are such $6$ points in $\mathbb{P}^2$ and 
the corresponding lines form a sixer. 
We call it \textit{the sixer corresponding to $M$}. 

\begin{lem}\label{det}
	Every determinantal representation $M$ of $X$ is equivalent to the following form: 
	\begin{equation}\label{eq:mat_CS}
		M_0 = 
		\begin{bmatrix}
			0 & -\pi'_2 & \pi_3 \\
			\pi_1 & 0 & -\pi'_3 \\
			-\pi'_1 & \pi_2 & 0 
		\end{bmatrix}, 
	\end{equation}
	where $\pi_i, \pi'_i$ $(i=1,2,3)$ are linear forms. 
\end{lem}
\begin{proof}
	Let $P_i = {}^t \! (a_{0i}, a_{1i},a_{2i})$ be points in $\mathbb{P}^2$ which give the sixer corresponding to $M$, 
	and $M' := M \, [ P_1 \ P_2 \ P_3 ]$, where $P_i$ are regarded as column vectors. 
	Since the $3$ planes $m'_{11}, m'_{21}, m'_{31}$ given by $M' \, {}^t \!(1,0,0) = 0$ share a line, 
	there exists a non-trivial equation $\alpha m'_{11} + \beta m'_{21} + \gamma m'_{31} = 0$. 
	Applying an invertible matrix whose first row is $(\alpha \  \beta \ \gamma)$ to $M'$ from left, 
	we have a matrix whose $(1,1)$-entry is zero. 
	In the same way to the points $(0,1,0)$ and $(0,0,1)$, 
	we get an equivalent determinantal representation which has zeros along diagonal $M_0$. 
	(See the proof of \cite[Proposition 3.3]{Buc07}.)
\end{proof}

\begin{rmk}
	We note that 
	$3$ skew lines of the sixer corresponding to the determinantal representation (\ref{eq:mat_CS}) 
	are on the points $(1,0,0)$, $(0,1,0)$, $(0,0,1)$ in $\mathbb{P}^2$. 
	The set of $3$ skew lines together with the one of the determinantal representation ${}^t \! M$ forms a half of double-six. 
	Because a half of double-six extends to the double-six uniquely, 
	the pair of sixeres corresponding to $M$ and ${}^t \! M$ is a double-six. 
\end{rmk}

The equation of $X$ given by the determinant of the matrix (\ref{eq:mat_CS}) is called a \textit{Cayley--Salmon equation}: 
\begin{equation}\label{eq:CS}
		\pi_1\pi_2\pi_3 - \pi'_1\pi'_2\pi'_3 = 0. 
\end{equation}
Giving a Cayley--Salmon equation $\det M_0 = 0$ of $X$, 
we have $9$ mutually distinct lines $\ell_{ij}$ defined by $\pi_{i} = \pi'_j = 0$. 
In fact, if $\ell_{i1} = \ell_{i'2}$, 
then the point $\pi_i = \pi_1' = \pi_j = 0 \ (j \neq i,i')$ is a singular point on $X$. 
So the planes $\pi_i = 0$ and $\pi'_i = 0$ ($i=1,2,3$) are tritangent planes of $X$, 
and the pair $(\{\pi_1,\pi_2,\pi_3\}, \{\pi'_1,\pi'_2, \pi'_3\})$ is a conjugate pair of triads of tritangent trios. 
Namely, a Cayley--Salmon equation (\ref{eq:CS}) of $X$ induces a conjugate pair of triads of tritangent trios.

We note that 
the conjugate pair of triad of tritangent trios given by the Cayley--Salmon equation $\det M_0 = 0$ 
contains the $6$ lines which form the half of double-six given by $M_0, {}^t \! M_0$. 
Conversely, for a conjugate pair of triad of tritangent trios, we can take $6$ half of double-six. 

\begin{lem}\label{CS:tri}
	Let $(\{\pi_1,\pi_2,\pi_3\},\{\pi_1',\pi_2',\pi_3'\})$ be a conjugate pair of triads of tritangent trios of $X$, 
	then $X$ is defined by the following: 
	\[ 
	\pi_1 \pi_2 \pi_3 + \lambda \pi_1' \pi_2' \pi_3' = 0, 
	\]
	for some $\lambda \in k\setminus \{0\}$. 
	Hence there is a one-to-one correspondence between 
	Cayley--Salmon equations of $X$ and conjugate pairs of triads of tritangent trios of $X$. 
	There are $120$ essentially distinct Cayley--Salmon equation of $X$. 
\end{lem}

\begin{proof}
	We write $\ell_{ij}$ for the line defined by $\pi_i = \pi'_j = 0$. 
	We take the $6$ lines which form a half of double-six 
	from the conjugate pair of triads of tritangent trios: 
	\begin{equation}\label{eq:ex_half}
	\begin{pmatrix}
		\ell_{11} & \ell_{22} & \ell_{33} \\
		\ell_{32} & \ell_{13} & \ell_{21} 
	\end{pmatrix}.
	\end{equation}
	We take the determinantal representation $M= (m_{ij})_{ij}$ 
	corresponding to the sixer of the first row of the extended double-six of (\ref{eq:ex_half}). 
	From Lemma \ref{det}, by exchanging $M$ with an equivalent one, we can assume that $M$ has zeros along the diagonal. 
	Since the half of double-six associated to $M, {}^t \! M$ is (\ref{eq:ex_half}), 
	the $2$ planes of each column and row intersect in the following line: 
	\[\begin{split}
	&\left[
	\begin{array}{ccc}
		0 & m_{12} & m_{13} \\
		m_{21} & 0 & m_{23} \\
		m_{31} & m_{32} & 0 \\	
	\end{array} \right] 
	\begin{array}{cc}
		\rightharpoonup & \ell_{32} \\
		\rightharpoonup & \ell_{13}  \\
		\rightharpoonup & \ell_{21}  \\
	\end{array} \\
	 & \qquad  \begin{array}{ccc}
	\downharpoonleft & \downharpoonleft &\downharpoonleft 	\\
	\ell_{11} & \ell_{22} & \ell_{33} \\
	\end{array}
\end{split} \]
	The plane defined by $m_{21} = 0$ contains the $2$ lines $\ell_{11}$ and  $\ell_{13}$, 
	hence it coincides with the plane $\pi_1 = 0$. 
	In the same way, 
	we see that 
	$m_{31} = \pi_1'$,  
	$m_{32} = \pi_2$,  
	$m_{12} = \pi_2'$, 
	$m_{13} = \pi_3$, and
	$m_{23} = \pi_3'$
	as planes. 
	So the determinant of $M$ which defines $X$ is the Cayley--Salmon equation 
	$\pi_1 \pi_2 \pi_3 + \lambda \pi_1' \pi_2' \pi_3' = 0$ for some $\lambda \neq 0$. 
\end{proof}

\section{The octanomial normal form}\label{sec:octa}
 
It is well-known that  in characteristic $\neq 2,3$ a general cubic surface is projectively isomorphic to a surface in $\mathbb{P}^4$ given by the equation: 
\begin{equation}\label{eq:sylvester}
	\sum_{i=0}^4 a_ix_i^3 = \sum_{i=0}^4 x_i = 0, 
\end{equation}
where $a_0,\dots ,a_4$ are parameters \cite[Corollary 9.4.2]{Dol12}. 
The parameters are uniquely determined up to permuting and scaling. 
The equation (\ref{eq:sylvester}) is called the \textit{Sylvester normal form}. 
We note that in characteristic $2$ or $3$ the Sylvester form is \textit{not} a normal form. 
In fact, all nonsingular surfaces defined by (\ref{eq:sylvester}) 
are projectively isomorphic to the Fermat cubic surface in characteristic $2$ \cite[\S 6]{DD19},
and are non-reduced in characteristic $3$.  
Even in characteristic $0$, 
there exist cubic surfaces which cannot be defined by Sylvester forms. 
In fact, 
since the automorphism group of a cubic surface defined by a Sylvester form is a subgroup of $\mathfrak{S}_5$, 
the Fermat cubic surface 
whose automorphism group is $(\mathbb{Z}/3\mathbb{Z})^3 \rtimes \mathfrak{S}_4$ 
is defined by no Sylvester form. 

Other normal forms are in Table \ref{tab:normal forms}, 
where the heading `Char' gives the conditions of characteristic and 
`Surface' gives the surfaces defined by the normal forms in a suitable coordinates. 
Here, $a_i, a_{ijk}$ are parameters and $\ell$ is a linear form in $x_0,\dots ,x_3$. 

\begin{table}[hb]
	\caption{Normal forms of cubic surfaces}\label{tab:normal forms}
	\begin{tabular}{llcc}
		\toprule
		Name & 
		Normal form & 
		Char & 
		Surface \\
		\midrule
		Sylvester \cite[Cor 9.4.2]{Dol12}&
		$\sum_{i=0}^4 a_ix_i^3 = \sum_{i=0}^4 x_i = 0 \ (\subset \mathbb{P}^4)$ &
		$\neq 2,3$ & 
		general \\
		Emch \cite[Thm 6.3]{DD19}\footnotemark[1] &
		$\sum_{i=0}^3 x_i^3 + \sum_{0\leq i < j < k\leq 3} a_{ijk} x_ix_jx_k = 0$ &
		\multicolumn{1}{|>{\columncolor[gray]{0.7}}c|}{any} & 
		general
		 \\
		 Dolgachev \cite[Cor 9.3.4]{Dol12}\footnotemark[2] & 
		 $x_0x_1x_2 + x_3(x_0 + x_1 + x_2 + x_3) \ell = 0$ &
		 $0$ & 
		 \multicolumn{1}{|>{\columncolor[gray]{0.7}}c|}{any} \\
		  Cremona  \cite[Thm 9.4.6]{Dol12}&
		  $\sum_{i=0}^5 x_i^3 = \sum_{i=0}^5 x_i = \sum_{i=0}^5 a_ix_i = 0 \ (\subset \mathbb{P}^5)$ &
		  $0$ &
		  \multicolumn{1}{|>{\columncolor[gray]{0.7}}c|}{any} \\
		  \bottomrule
	\end{tabular}
\end{table}

\footnotetext[1]{
	Dolgachev and Duncan conjectured that in characteristic $2$
	every cubic surface is written in an Emch normal form \cite[Remark 6.5]{DD19}.}
\footnotetext[2]{
	In characteristic $2$, 
	the Fermat cubic surface is not defined by the Dolgachev normal form. 
	In fact, any surface defined this normal form has a \textit{canonical point} (See \cite[Proposition 4.4]{DD19}). 
}

M. Panizzut, E-C. Sert\"{o}z, and B. Sturmfels discovered a new normal form called the \textit{octanomial form} \cite{Pan20}. 
They showed that every (not necessary nonsingular) cubic surface given by blowing up $6$ points in $\mathbb{P}^2$ was projectively isomorphic to a surface defined by the equations: 
\begin{equation}\label{eq:octa2}
	 a_0x_1x_2x_3 + a_1x_0x_2x_3 + a_2x_0x_1x_3 + a_3x_0x_1x_2 
	 + a_4x_0^2x_1 + a_5x_0x_1^2 + a_6x_2^2x_3 + a_7x_2x_3^2 = 0, 
\end{equation} 
where $a_0,\dots ,a_7$ are parameters. 
They presented this normal form as a suitable one for $p$-adic geometry, 
as it revealed the intrinsic del Pezzo combinatorics of the $27$ trees in the tropicalization. 
Their methods to introduce and to investigate this normal form were heavily computational. 

In this section, 
we will show that  in \textit{any} characteristic \textit{every nonsingular} cubic surface (which is given by blowing up some $6$ points in $\mathbb{P}^2$ \textit{in general position}) is projectively isomorphic to a surface defined by an octanomial form (\ref{eq:octa2}) with $a_4 = a_5 = a_6 = a_7 = 1$ (Theorem \ref{octa}). 
Our method to prove this is more conceptual and computation-free. 

\begin{proof}[Proof of Theorem \ref{octa}]
	Let $X$ be a cubic surface. 
	It follows from Lemma \ref{CS:tri} that $X$ is defined by the equation 
	$f=\pi_1\pi_2\pi_3 + \pi'_1\pi'_2\pi'_3$, 
	where $\{\pi_1,\pi_2,\pi_3\},\{\pi'_1,\pi'_2,\pi'_3\}$ is a conjugate pair of triads of tritangent trios of $X$. 
	We claim that the $4$ planes $\pi_1, \pi_2, \pi'_1, \pi'_2$ are linearly independent. 
	In fact, if there exist $a,b,c \in k$ such that $\pi'_2 = a \pi_1 + b \pi_2 + c \pi'_1$, 
	the $2$ lines $\ell_{21}, \ell_{12}$ intersect at the point $\pi_1=\pi_2=\pi'_1=0$, 
	which is a contradiction. 
	
	By applying the coordinate change 
	$(\pi_1,\pi_2,\pi_1',\pi_2') \mapsto (x_0,x_1,x_2,x_3)$, 
	we have 
	\[ f= x_0x_1(a_0x_0 + a_1x_1 + a_2x_2 + a_3x_3) + x_2x_3(b_0x_0+b_1x_1+b_2x_2+b_3x_3). \]
	Moreover, if $b_3=0$, the point $(0,0,0,1)$ is a singular point of $X$, 
	hence $b_3\neq 0$. 
	In the same way, we see that $a_0,a_1,b_2 \neq 0$.  
	Then by scaling $a_0^{-1}, a_1^{-1}, b_2^{-1}, b_3^{-1}$ with $x_0,x_1,x_2,x_3$ respectively, 
	we have the equation 
	\[ f=\frac{1}{a_0a_1} x_0x_1 \left(x_0+x_1+\frac{a_2}{b_2}x_2 + \frac{a_3}{b_3}x_3 \right)
	+ \frac{1}{b_2b_3} x_2x_3 \left( \frac{b_0}{a_0}x_0 + \frac{b_1}{a_1}x_1 + x_2 + x_3 \right). \]
	Finally scaling $(a_0a_1)^{1/3}$ with $x_0, x_1$, and $(b_2b_3)^{1/3}$ with $x_2, x_3$ respectively, 
	we have the octanomial form (\ref{eq:octa}) of $X$. 
\end{proof}

\begin{rmk}
	In the original octanomial normal form (\ref{eq:octa2}), 
	the parameters $a_4,\dots ,a_7$ are not zero if and only if 
	the surface is nonsingular \cite[Proposition 2.1]{Pan20}. 
 \end{rmk}

\begin{rmk}\label{octa_para}
	Thanks to Theorem \ref{octa}, 
	there is the surjection $\Phi\colon \mathbb{A}_k^4 \to \mathcal{M}(k)$, 
	where $\mathbb{A}_k^4$ is the octanomial parameter space. 
	We compute the degree of this. 
	To be precise, we show that
	the number of octanomial parameters of a cubic surface $X$ is 
	$25920/ |\Aut (X)|$ in $p\neq 3$, $8640/ |\Aut (X)|$ in $p=3$. 
	
	Referring to the proof of Theorem \ref{octa}, 
	we can compute this as follows. 
	A cubic surface $X$ has $120$ Cayley--Salmon equations (Lemma \ref{CS:tri}) 
	corresponding to conjugate pairs of triads of tritangent trios $(\{\pi_1,\pi_2,\pi_3\}, \{\pi_1',\pi_2',\pi_3'\})$, 
	and $72 = 3!\times 3!\times 2$ ways 
	to choose $4 = 2 + 2$ tritangent planes $\pi_i, \pi_j, \pi'_{i'}, \pi'_{j'}$ 
	which are changed to $x_0,x_1,x_2,x_3$ respectively. 
	There are $3$ scaling ways in $p\neq 3$. 
	If $X$ does not have a non-trivial automorphism, 
	no redundancy occur, 
	and the number of octanomial parameters is 
	$25920 = 120 \times 72 \times 3$ in $p\neq 3$, $8640 = 120 \times 72$ in $p=3$. 
	
	We assume that $X$ does not admit an automorphism of type $3C$. 
	Let $g$ be a non-trivial automorphism.  
	By the orbits of the $27$ lines  of each type we will show in the next chapter, 
	there is no conjugate pair of triads of tritangent trios which consists of $g$-invariant lines. 
	The octanomial parameters given by 
	$(\pi_1,\pi_2,\pi'_1,\pi'_2) \mapsto (x_0,x_1,x_2,x_3)$ and 
	$(g\pi_1, g\pi_2, g\pi_1', g\pi_2') \mapsto (x_0',x_1',x_2',x_3')$ are the same 
	because the representation of $g$ with respect to the coordinates $(x_0,\dots ,x_3)$ and $(x_0',\dots ,x_3')$ 
	is the identity. 
	As a result, the $25920$ (resp. $8640$ in $p=3$) octanomial parameters are duplicated $|\Aut (X)|$ times. 
	
	An automorphism of type $3C$ has a conjugate pair of triad of tritangent trios consisting of invariant lines. 
	In the case, the cubic surface is isomorphic to the Fermat cubic surface, 
	and we can directly compute all octanomial parameters (Remark \ref{octa_para_3C}). 
\end{rmk}

\section{The octanomial normal forms of the strata}\label{sec:main}

In this section, 
we will give the octanomial normal forms of the strata and automorphisms of the types, 
in descending order of dimension (Theorem \ref{main thm}). 
As we see in the above, there are many octanomial parameters of a cubic surface up to projective equivalence. 
We give one of these octanomial parameters of each stratum which preserve most specialization of parameters. 
We construct them 
by recalling the arrangement of Eckardt points given by \cite{DD19} 
or making coordinate change directly. 

\subsection{The stratum $2A$}\label{subsec:2A}

The stratum $2A$ is of dimension $3$ in $\mathcal{M}(k)$. 
An automorphism of type $2A$ is the most fundamental 
because any cubic surface with nontrivial automorphism group has ones of type $2A$,  
and because automorphisms of most type are realized by products of some ones of type $2A$. 
It is known that an automorphism of order $2$ is of type $2A$ if and only if 
its action on $\mathbb{P}^3$ is given by  
\[ g= \begin{bmatrix}
	0 & 1 & 0 & 0 \\
	1 & 0 & 0 & 0 \\
	0 & 0 & 1 & 0 \\
	0 & 0 & 0 & 1 
\end{bmatrix},   \]
in a suitable coordinates. 
The locus of fixed points of this action consists of 
a hyperplane $A = \kernel (g-\id)$ 
and a unique point $q_0 = \image (g-\id)$, 
which are called the \textit{axis} and the \textit{center} of $g$, respectively. 
The center is not on the axis in $p\neq 2$, 
but always on it in $p=2$. 

Let $g$ be an automorphism of a cubic surface of type $2A$. 
The $g$-orbits of the $27$ lines are the following: 
\[ (E_6), (F_{56}), (G_5), (G_i, F_{i6}), (E_5, G_6), (E_i, F_{i5}), (F_{ij}, F_{kl}),  \]
where $\{i,j,k,l\} = \{1,2,3,4\}$. 
The $3$ invariant lines form a tritangent trio. 
The remaining $24$ lines are partitioned into $12$ pairs 
where each pair together with one of the invariant lines form a tritangent trio. 
See \cite[Proposition 9.1]{DD19} for details. 

When the $3$ lines which form a tritangent trio of a cubic surface intersect at a point, 
we call the point an \textit{Eckardt point}. 
In Subsection \ref{subsec:3C}, 
we will see the arrangement of Eckardt points of 
the Fermat cubic surface.

The $3$ invariant lines of an automorphism of type $2A$ intersect at a point, which is an Eckardt point. 

\begin{prp}[{\cite[Theorem 9.2]{DD19}}]\label{DD19:9.2}
	Let $X$ be a cubic surface. 
	The center of an  automorphism of $X$ of type $2A$ is an Eckardt point on $X$. 
	Conversely, any Eckardt point on $X$ is the center of an automorphism of $X$ of type $2A$. 
\end{prp}

\begin{thm}[The $2A$-octanomial normal form]\label{2A-octa}
	Every cubic surface admitting an automorphism of type $2A$ is 
	projectively isomorphic to the surface given by the octanomial form with $a_0 = a_1$: 
	\begin{equation}\label{eq:2A-octa}
		x_0x_1(x_0+x_1+ a_3x_2 + a_2x_3 ) + x_2x_3(a_0x_1 + a_0x_0 + x_2+x_3) = 0.
	\end{equation}
	The surfaces given by the above have the automorphism of type $2A$: 
	\begin{equation}\label{eq:2A-auto}
	\begin{bmatrix}
		0 & 1 & 0 & 0 \\
		1 & 0 & 0 & 0 \\
		0 & 0 & 1 & 0 \\
		0 & 0 & 0 & 1\\
	\end{bmatrix}. 
	\end{equation}
	The corresponding Eckardt point is $(1,-1,0,0)$ 
	on the tritangent plane $x_0 + x_1 + a_3x_2 + a_2x_3 = 0$, 
	and the axis is $x_0 - x_1 = 0$. 
\end{thm}
\begin{proof}
	We observe that the automorphism (\ref{eq:2A-auto}) of the surface (\ref{eq:2A-octa}) is of type $2A$. 
	
	Conversely, let $X$ be a surface which admits an automorphism $g$ of type $2A$. 
	It follows from Proposition \ref{DD19:9.2} that there exists an Eckardt point. 
	Let $\pi_1$ be the corresponding tritangent plane, 
	and we take a conjugate pair of triads of tritangent trios $(\{\pi_1,\pi_2,\pi_3\}, \{\pi'_1,\pi'_2, \pi'_3\})$ 
	which contains $\pi_1$. 
	Then $X$ is defined by the equation 
	$\pi_1\pi_2\pi_3 + \pi'_1\pi'_2\pi'_3 = 0$ by Lemma \ref{CS:tri}. 
	By coordinate change, 
	we have 
	$\pi_2 =x_0, \ 
	\pi_3 = x_1, \ 
	\pi'_2=x_2, \  
	\pi'_3 = x_3$ and 
	$\pi_1 = x_0 + x_1 + a_3x_2 + a_2x_3, \ 
	\pi'_1 = a_1x_0 + a_0x_1 + x_2 + x_3$ 
	(See the proof of Theorem \ref{octa}).
	Because the tritangent plane $\pi_1$ corresponds to an Eckardt point, 
	the $3$ lines on $\pi_1$ given by the intersection with the $3$ planes 
	$x_j = 0 \ (j=2,3)$ and $a_1x_0 + a_0x_1 + x_2 + x_3 = 0$ 
	intersect at a point, 
	so we have $a_0 = a_1$ and the point is $(1,-1,0,0)$. 
\end{proof}

\subsection{The stratum $2B$}
The stratum $2B$ is of dimension $2$ in $\mathcal{M}(k)$. 
The action of an automorphism of type $2B$ on $\mathbb{P}^3$ is given by 
\[ \begin{bmatrix}
	0 & 1 & 0 & 0 \\
	1 & 0 & 0 & 0 \\
	0 & 0 & 0 & 1 \\
	0 & 0 & 1 & 0 
\end{bmatrix},  \]
in a suitable coordinates. 
An automorphism of type $2B$ is distinguished from one of type $2A$ 
by the dimension of the fixed locus of the action on $\mathbb{P}^3$. 

Let $g$ be an automorphism of a cubic surface of type $2B$. 
The $g$-orbits of the the $27$ lines are the following: 
\[ (E_6), (F_{i6}), (G_i), (G_{m}, F_{m6}), (E_1, E_2), (G_6, F_{12}), (E_i, F_{jk}), (F_{mi},F_{ni}), \]
where $\{m,n\} = \{1,2 \}, \{i,j,k\} = \{3,4,5\}$. 
There is the unique pointwise-fixed line $E_6$ which intersects with each of $6$ invariant lines. 
The remaining $20$ lines are partitioned into $10$ pairs. 
Two of them form tritangent trios together with the pointwise-fixed line, 
 and the others are pairs of mutually skew lines.    
 See \cite[Proposition 9.1]{DD19} for details. 

\begin{lem}[{\cite[Lemma 9.5]{DD19}}]\label{DD19:9.5}
	Let $X$ be a cubic surface admitting an automorphism $g$ of type $2B$, 
	$\ell_0$ be the unique line pointwise-fixed by $g$. 
	In characteristic $\neq 2$, 
	there are exactly $2$ Eckardt points on $\ell_0$ 
	whose corresponding automorphisms of type $2A$ generate a group isomorphic to $(\mathbb{Z}/2\mathbb{Z})^2$ 
	containing $g$. 
	In characteristic $2$, 
	$\ell_0$ contains $5$ Eckardt points 
	whose corresponding automorphisms of type $2A$ generate a group isomorphic to $(\mathbb{Z}/2\mathbb{Z})^4$ 
	containing $g$. 
\end{lem}

\begin{thm}[The $2B$-octanomial normal form]\label{2B-octa}
	Every cubic surface admitting an automorphism of type $2B$ is projectively isomorphic to 
	the surface given by the octanomial form with $a_0=a_1$, $a_2 = a_3$: 
	\begin{equation}\label{eq:2B-octa}
		x_0x_1(x_0+x_1 + a_2x_2 + a_2x_3) + x_2x_3(a_0x_0 + a_0x_1 + x_2 + x_3) = 0.
	\end{equation}
	The surfaces given by the above have the automorphism of type $2B$: 
	\begin{equation}\label{eq:2B-auto}
	\begin{bmatrix}
		0 & 1 & 0 & 0 \\
		1 & 0 & 0 & 0 \\
		0 & 0 & 0 & 1 \\
		0 & 0 & 1 & 0 
	\end{bmatrix}. 
	\end{equation}
	The pointwise-fixed line is 
	$a_0x_0 + a_0x_1 + x_2 + x_3 = x_0 + x_1 + a_2x_2 + a_2x_3 = 0$ 
	containing the $2$ Eckardt points $(1,-1,0,0)$ and $(0,0,1,-1)$. 
\end{thm}

\begin{proof}
	The surface defined by (\ref{eq:2B-octa}) has the line 
	$x_0 + x_1 + a_2x_2 + a_2x_3 = a_0x_0 + a_0x_1 + x_2 + x_3 = 0$ 
	containing the $2$ Eckardt points $(1,-1,0,0), (0,0,1,-1)$. 
	The corresponding automorphisms of type $2A$ are 
	\[ 
	\begin{bmatrix} 
	0 & 1 & 0 & 0 \\
	1 & 0 & 0 & 0 \\
	0 & 0 & 1 & 0 \\
	0 & 0 & 0 & 1
	\end{bmatrix}, \quad 
	\begin{bmatrix}
	1 & 0 & 0 & 0 \\
	0 & 1 & 0 & 0 \\
	0 & 0 & 0 & 1 \\
	0 & 0 & 1 & 0 
	\end{bmatrix}. 
	\] 
	The product (\ref{eq:2B-auto}) is an automorphism of type $2B$. 
		
	Conversely, let $X$ be a cubic surface admitting an automorphism $g$ of type $2B$. 
	There is a line $\ell$ on $X$ containing $2$ Eckardt points by Lemma \ref{DD19:9.5}. 
	We take the $2$ tritangent planes $\pi_1, \pi'_1$ corresponding to the Eckardt points on $\ell$, 
	then take a conjugate pair of triads of tritangent trios $\{\pi_1,\pi_2,\pi_3\}, \{\pi'_1,\pi'_2,\pi'_3\}$ 
	containing $\pi_1$ and $\pi'_1$. 
	It follows from Lemma \ref{CS:tri} that $X$ can be defined by $\pi_1\pi_2\pi_3 + \pi'_1\pi'_2\pi'_3 = 0$. 
	By coordinate change, 
	we can put  
	$\pi_2 = x_0, \ 
	\pi_3 = x_1, \ 
	\pi'_2 = x_2, \ 
	\pi'_3 = x_3$ and 
	$\pi_1 = x_0 + x_1 + a_3x_2 + a_2x_3, \ 
	\pi'_1 = a_1x_0 + a_0x_1 + x_2 + x_3$
	(See the proof of Theorem \ref{octa}). 
	Because $\pi_1$ and $\pi'_1$ are tritangent planes corresponding to Eckardt points, 
	we have $a_0 = a_1$ and $a_2 = a_3$. 
\end{proof}

\subsection{The stratum $3D$}
The stratum $3D$ is of dimension $2$ in $\mathcal{M}(k)$. 
The action of an automorphism of type $3D$ on $\mathbb{P}^3$ is given by 
\[ \begin{bmatrix}
	1 & 0 & 0 & 0 \\
	0 & 0 & 1 & 0 \\
	0 & 0 & 0 & 1 \\
	0 & 1 & 0 & 0 
\end{bmatrix}, \]
up to projective equivalence \cite[Theorem 10.4]{DD19}. 

\begin{rmk}[Automorphisms of order $3$ {\cite[Subsection 10.2]{DD19}}]\label{ord3}
	Let $X$ be a cubic surface admitting an automorphism $g$ of order $3$. 
	The $g$-orbit of a line on $X$ is 
	an invariant line, a tritangent trio, or a skew triple of lines.  
	The automorphism $g$ of order $3$ is of type $3A$, $3C$, or $3D$. 
	Of whichever type $g$ is, 
	the $27$ lines on $X$ are partitioned into $3$ $g$-invariant conjugate pairs of triads of tritangent trios, 
	for example: 
	\begin{equation}\label{eq:ord3}
		\begin{matrix}
			F_{14} & F_{25} & F_{36} \\
			F_{26} & F_{34} & F_{15} \\
			F_{35} & F_{16} & F_{24} 
		\end{matrix}
		\hspace{15mm}
		\begin{matrix}
			E_1 & G_2 & F_{12} \\
			G_3 & F_{23} & E_2 \\
			F_{13} & E_3 & G_1 
		\end{matrix}
		\hspace{15mm}
		\begin{matrix}
			E_4 & G_5 & F_{45} \\
			G_6 & F_{56} & E_5 \\
			F_{46} & E_6 & G_4 
		\end{matrix}.
	\end{equation}
	Since $g$ acts on triads of tritangent trios, 
	$g$ induces a $g$-invariant conjugate pair of triads of tritangent trios 
	to permute rows or columns, or to transpose the matrix.  
	So the $g$-orbits of the $9$ lines in a $g$-invariant conjugate pair of triads of tritangent trios consist of 
	 $9$ invariant lines, 
	 $3$ of tritangent trios, or 
	 $3$ of skew triples. 
\end{rmk}

Let $g$ be an automorphism of a cubic surface of type $3D$. 
In a suitable labeling, the $g$-orbits of $27$ lines are 
given by 
$(E_6, G_5,F_{56})$, $(G_4, F_{45}, E_5)$, $(F_{46},E_4,G_6)$, 
$(F_{14}, F_{16}, F_{15})$, $(F_{26},F_{25},F_{24})$, $(F_{35}, F_{34}, F_{36})$, 
$(E_1,F_{23},G_1)$, $(G_3, E_3, F_{12})$, $(F_{13},G_2,E_2)$, 
 preserving the $3$ conjugate pairs of triads of tritangent trios (\ref{eq:ord3}). 
 There are $6$ skew triples and $3$ tritangent trios. 

A \textit{trihedral line} of a cubic surface is a line in $\mathbb{P}^3$ 
which is not on the surface 
and contains exactly $3$ Eckardt points of the surface.  
From the next lemma, 
if a line in $\mathbb{P}^3$ which is not on the cubic surface contains distinct $2$ Eckardt points, 
it is a trihedral line.   

\begin{lem}\label{DD19:10.1}
	Let $X$ be a cubic surface, 
	$(\{\pi_1, \pi_2, \pi_3\}, \{\pi'_1,\pi'_2,\pi'_3\})$ be a conjugate pair of triads of tritangent trios. 
	The following are equivalent: 
	\begin{enumerate}
	\item[(i)]
	$\ell = \pi_1'  \cap \pi_2' \cap \pi_3'$ is a line; 
	\item[(ii)]
	all of the tritangent planes $\pi_1,\pi_2, \pi_3$ correspond to Eckardt points; 
	\item[(iii)]
	two of the tritangent planes $\pi_1,\pi_2, \pi_3$ correspond to Eckardt points. 
	\end{enumerate}
	Satisfying the above condition, 
	the line $\ell$ is a trihedral line. 
	Conversely, 
	let $\ell$ be a trihedral line, 
	then there is a triad of tritangent trios $\{\pi'_1, \pi'_2, \pi'_3\}$ such that $\ell =  \pi'_1\cap \pi'_2 \cap \pi'_3$. 
\end{lem}
 
\begin{proof}
	For the former, see \cite[Lemma 10.1]{DD19}. 
	For the converse, 
	let $\ell$ be a trihedral line and $\pi_1, \pi_2$ be 
	the tritangent planes corresponding to the Eckardt points $q_1,q_2$ on $\ell$. 
	If $\pi_1$ and $\pi_2$ share a line on $X$, 
	the line contains $q_1$ and $q_2$, 
	so it is exactly $\ell$, 
	which is a contradiction. 
	There is a unique tritangent plane $\pi_3$ such that $\{\pi_1, \pi_2, \pi_3\}$ is a triad of tritangent trios. 
	Let $\{\pi'_1, \pi'_2, \pi'_3\}$ be the conjugate triad of tritangent trios, 
	then $\ell = \pi'_1 \cap \pi'_2 \cap \pi'_3$. 
\end{proof}

\begin{lem}[{\cite[Lemma 10.15]{DD19}}]\label{DD19:10.15}
	Let $X$ be a cubic surface admitting an automorphism of type $3D$. 
	Then there is a $g$-invariant trihedral line $\ell$ 
	such that the automorphisms of type $2A$ corresponding to the Eckardt points on $\ell$ 
	generate a group isomorphic to $\mathfrak{S}_3$ containing $g$.  
\end{lem}

\begin{thm}[The $3D$-octanomial normal form]\label{3D-octa}
	Every cubic surface admitting an automorphism of type $3D$ is projectively isomorphic to 
	the surface given by the octanomial form with $a_0=a_1=0$: 
	\begin{equation}\label{eq:3D-octa}
		 x_0x_1(x_0 + x_1 + a_3x_2 + a_2x_3) + x_2x_3(x_2+x_3) = 0.
	\end{equation}
	The surfaces given by the above have the automorphism of type $3D$: 
	\begin{equation}\label{eq:3D-auto}
	\begin{bmatrix}
		0 & 1 & 0 & 0 \\
		-1 & -1 & -a_3 & -a_2 \\
		0 & 0 & 1 & 0 \\
		0 & 0 & 0 & 1 
	\end{bmatrix}. 
	\end{equation}
	The invariant trihedral line is $x_2 = x_3 = 0$, 
	containing $3$ Eckardt points 
	$(0,1,0,0)$, $(1,0,0,0)$ and $(1,-1,0,0)$. 
\end{thm}

\begin{proof}
	The surface defined by (\ref{eq:3D-octa}) has the triad of tritangent trios: 
	$x_0=0, \ x_1=0, \ x_0+x_1+a_3x_2 +a_2x_3=0$
	which corresponds to the Eckardt points 
	$(0,1,0,0), \ (1,0,0,0),\ (1,-1,0,0)$ respectively. 
	Its conjugate triad of triad of tritangent trios is 
	$x_2 = 0, \ x_3 = 0, \  x_2 + x_3 = 0$, 
	whose intersection $x_2=x_3=0$ is a trihedral line 
	by Lemma  \ref{DD19:10.1}. 
	The automorphisms of type $2A$ corresponding to the above Eckardt points are 
	\[ \begin{bmatrix}
		1 & 0 & 0 & 0 \\
		-1 & -1 & -a_3 & -a_2 \\
		0 & 0 & 1 & 0 \\
		0 & 0 & 0 & 1
	\end{bmatrix} , \quad 
	\begin{bmatrix}
		-1 & -1 & -a_3 & -a_2 \\
		0 & 1 & 0 & 0 \\
		0 & 0 & 1 & 0 \\
		0 & 0 & 0 & 1
	\end{bmatrix} ,\quad 
	\begin{bmatrix}
		0 & 1 & 0 & 0 \\
		1 & 0 & 0 & 0 \\
		0 & 0 & 1 & 0 \\
		0 & 0 & 0 & 1 
	\end{bmatrix},\]
	respectively. 
	It follows from Lemma \ref{DD19:10.15} that 
	they generate a group isomorphic to $\mathfrak{S}_3$, 
	and the products of two of them 
	\[ \begin{bmatrix} 
		0 & 1 & 0 & 0 \\
		-1 & -1 & -a_3 & -a_2 \\
		0 & 0 & 1 & 0 \\
		0 & 0 & 0 & 1 
	\end{bmatrix}, \quad 
	\begin{bmatrix}
		-1 & -1 & -a_3 & -a_2 \\
		1 & 0 & 0 & 0 \\
		0 & 0 & 1 & 0 \\
		0 & 0& 0& 1
	\end{bmatrix} \]
	are automorphisms of type $3D$. 

	Conversely, let $X$ be a cubic surface admitting an automorphism $g$ of type $3D$. 
	By Lemma \ref{DD19:10.15},
	there is a $g$-invariant trihedral line $\ell$, 
	and there is a conjugate pair of triads of tritangent trios $(\{\pi_1,\pi_2,\pi_3\},\{\pi_1',\pi_2',\pi_3'\})$ 
	with all of $\pi_1,\pi_2,\pi_3$ corresponding to Eckardt points 
	by Lemma \ref{DD19:10.1}. 
	Lemma \ref{CS:tri} implies that 
	$X$ can be defined by $\pi_1\pi_2\pi_3 + \pi_1' \pi_2' \pi_3' = 0$. 
	By applying coordinate change, 
	we may assume 
	$\pi_2 = x_0, \pi_3 = x_1, \pi'_2 = x_2, \pi'_3 = x_3$, 
	and we get an octanomial form (\ref{eq:octa}) by scaling. 
	Because the tritangent plane $\pi_2$ which is defined by $x_0=0$ corresponds to an Eckardt point, 
	the three lines on this plane given by the intersection with the planes 
	$x_j = 0 \ (j=2,3)$ and $a_1x_0 + a_0x_1 + x_2 + x_3 = 0$ 
	intersect at a point, so $a_0 = 0$. 
	In the same way, we have $a_1 = 0$. 
\end{proof}

\subsection{The stratum $3A$} 
The stratum $3A$ is of dimension $1$ in $\mathcal{M}(k)$. 
The action of an automorphism of type $3A$ on $\mathbb{P}^3$ is given by 
\[ 
\begin{bmatrix}
	\zeta_3 & 0 & 0 & 0 \\
	0 & 1 & 0 & 0 \\
	0 & 0 & 1 & 0 \\
	0 & 0 & 0 & 1
\end{bmatrix} \ 
(p\neq 3), \quad 
\begin{bmatrix}
	1 & 1 & 0 & 0 \\
	0 & 1 & 0 & 0 \\
	0 & 0 & 1 & 0 \\
	0 & 0 & 0 & 1 
\end{bmatrix} \ 
(p=3), 
\]
up to projective equivalence \cite[Theorem 10.4]{DD19}. 
The locus of fixed points of this action consists of a hyperplane $A$, called the \textit{axis}. 
An automorphism of a cubic surface of order $3$ has a pointwise-fixed plane in $\mathbb{P}^3$ 
if and only if it is of type $3A$. 

A cubic surface admits an automorphism of type $3A$ 
if and only if it is \textit{cyclic}
(i.e. it has a triple cover of $\mathbb{P}^2$ with cyclic Galois group) \cite[Lemma 10.5]{DD19}. 
In $p\neq 3$, 
any surface which admits an automorphism of type $3A$ is defined by $x_3^3 + g(x_0,x_1,x_2) = 0$, 
where $g$ is a cubic form in $x_0,x_1,x_2$. 

Let $g$ be an automorphism of a cubic surface of type $3A$. 
In a suitable labeling, 
the $g$-orbits of the $27$ lines are the following: 
$(F_{14},F_{25},F_{36})$, $(F_{26},F_{34},F_{15})$, $(F_{35},F_{16},F_{24})$, 
$(E_1, G_2, F_{12})$, $(G_3,F_{23},E_2)$, $(F_{13}, E_3, G_1)$, 
$(E_4, G_5, F_{45})$, $(G_6, F_{56}, E_5)$, $(F_{46}, E_6, G_4)$, 
preserving the $3$ conjugate pairs of triads of tritangent trios (\ref{eq:ord3}). 
There are $9$ invariant tritangent trios. 
The $3$ lines of each $g$-invariant tritangent trio intersect at a point on the pointwise-fixed plane $A$, 
so $X$ has exactly $9$ Eckardt points which are on $A$.  

\begin{thm}[The $3A$-octanomial normal form]\label{3A-octa}
	In $p \neq 3$ (resp. $p=3$), 
	a general cubic surface admitting an automorphism of type $3A$ is 
	projectively isomorphic to the surface given by the octanomial form 
	with $a_0 = a_1 = 0, \ a_2 + \zeta_3 a_3 = 0$ 
	(resp. $a_0 = a_1 = 0, \ a_2 + a_3 = 0$): 
	\begin{equation}\label{eq:3A-octa}
		 x_0x_1(x_0 + x_1 + a_3 x_2 - \zeta_3 a_3 x_3) + x_2x_3(x_2 + x_3) = 0 
	\end{equation}
	\begin{equation}\label{eq:3A-octa_char3}
		(\text{resp. } x_0x_1(x_0 + x_1 + a_3 x_2 - a_3 x_3) + x_2x_3(x_2 + x_3) = 0 ).
	\end{equation}
	The surfaces given by the above have the automorphism of type $3A$: 
	\begin{equation}\label{eq:3A-auto}
	\begin{bmatrix}
		1 & 0 & 0 & 0 \\
		0 & 1 & 0 & 0 \\
		0 & 0 & 0 & \zeta_3^2 \\
		0 & 0 & -\zeta_3^2 & -\zeta_3^2 
	\end{bmatrix} \quad
	\left(\text{resp. }
	\begin{bmatrix}
		1 & 0 & 0 & 0 \\
		0 & 1 & 0 & 0 \\
		0 & 0 & -1 & -1 \\
		0 & 0 & 1 & 0 
	\end{bmatrix}
	\right).
	\end{equation}
	The axis is defined by 
	$\zeta_3x_2 - x_3 =0$ 
	(resp. $x_2 -x_3 = 0$). 
\end{thm}

\begin{proof}
	In $p \neq 3$ 
	(resp. $p=3$), 
	by the coordinate change 
	$(x_2, x_3) \mapsto (x_2 + \zeta_3 x_3,  \zeta_3 x_2 + x_3)$
	(resp. $x_3 \mapsto x_2 - x_3$), 
	the surface (\ref{eq:3A-octa}) (resp. (\ref{eq:3A-octa_char3})) is defined by the equation 
	\[ 
	-x_3^3 -x_2^3 + a_3 (1-\zeta_3^2)x_0x_1x_2 + x_0x_1(x_0 + x_1) = 0 
	\]
	\[
	(\text{resp.} -x_2^3 + x_2x_3^2 + a_3 x_0x_1x_3 + x_0x_1(x_0 + x_1) =0), 
	\]
	which has the automorphism 
	$x_3 \mapsto \zeta_3 x_3$
	(resp. $x_2 \mapsto x_2 + x_3$) of type $3A$. 
\end{proof}

\subsection{The stratum $4A$ in $p \neq 2$}
The stratum $4A$ is of dimension $1$ in $\mathcal{M}(k)$. 
The action of an automorphism of type $4A$ on $\mathbb{P}^3$ is given by 
\[ 
\begin{bmatrix}
	i & 0 & 0 & 0 \\
	0 & -1 & 0 & 0 \\
	0 & 0 & 1 & 0 \\
	0 & 0 & 0 & 1
\end{bmatrix} \ (p\neq 2), \quad 
\begin{bmatrix}
	1 & 1 & 0 & 0 \\
	0 & 1 & 1 & 0 \\
	0 & 0 & 1 & 0 \\
	0 & 0 & 0 & 1
\end{bmatrix} \ (p=2), 
\]
up to projective equivalence \cite[Lemma 11.1]{DD19}. 

Let $g$ be an automorphism of a cubic surface of type $4A$. 
The $g$-orbits of the $27$ lines are the following: 
$(E_6), (F_{56}), (G_5)$, 
$(G_1, G_2, F_{16}, F_{26})$, $(G_3, G_4, F_{36}, F_{46})$, $(E_5, F_{24}, G_6, F_{13})$, 
$(E_1, F_{45}, F_{15}, E_4)$, $(F_{12}, F_{23}, F_{34}, F_{14})$, $(E_2, E_3, F_{25}, F_{35})$ in a suitable labeling. 
There are $3$ invariant lines which form a tritangent trio. 
The remaining $24$ lines are partitioned into $6$ orbits of $4$ lines. 
Each orbit consists of $2$ pairs of incident lines, 
each of which form a tritangent trio with one of the invariant lines. 
We note the $g^2$-orbits are the ones of type $2A$ in Subsection \ref{subsec:2A}. 

In $p=2$, 
an automorphism of type $4A$ is realized as products of automorphisms of type $2A$, 
and the stratum $4A$ is the same as the stratum $6E$.  
We will observe it in Subsection \ref{sec:6E=4B=4A_char2}.
On the other hand, 
an automorphism of type $4A$ is not realized as products of automorphisms of type $2A$ in $p\neq 2$. 
In the case, 
the arrangement of the Eckardt points has little information on the surface, 
hence we observe the normal form given by \cite{DD19}. 

\begin{lem}[{\cite[Lemma 11.4]{DD19}}]
	In $p \neq 2$, 
	a general cubic surface admitting an automorphism of type $4A$ is 
	projectively isomorphic to the surface defined by 
	\begin{equation}\label{eq:4A}
		x_3^2 x_2 + x_2^2 x_0 + x_1(x_1-x_0)(x_1-cx_0) = 0, 
	\end{equation}
	where $c$ is a parameter, 
	and the automorphism $(x_0, x_1, x_2, x_3) \mapsto (x_0, x_1, -x_2, ix_3)$ is of type $4A$. 
\end{lem}

\begin{thm}[The $4A$-octanomial normal form in $p \neq 2$]\label{4A-octa}
	In $p \neq 2$, 
	a general cubic surface admitting an automorphism of type $4A$ is 
	projectively isomorphic to the surface given by the octanomial form satisfying the condition $(\ast)$: 
	\begin{equation}\label{eq:4A-octa}
		x_0 x_1( x_0 + x_1 + a_3x_2 + a_2x_3 ) + x_2 x_3( a_0x_0 + a_0x_1 + x_2 + x_3 ) = 0.
	\end{equation}
	$(\ast)$: \ 
	$a_0 = a_1$, 
	$q^3 + 4pq + 8 -2a_0(3q^2 -4a_0q + 4p) = 0$ and 
	$3q -p^2 + a_0 (2a_0^2q - a_0(q^2 + 2p) + pq + 2) = 0$, 
	where $p=a_2 + a_3$, $q=a_2a_3$. 
	
	The surfaces given by the above have the automorphism of type $4A$: 
	\begin{equation}\label{eq:4A-auto}
	\begin{bmatrix}
		A + \frac{i-1}{2} & A-\frac{i+1}{2} & a_3(A-1) & a_2(A-1) \\
		A-\frac{i+1}{2} & A + \frac{i-1}{2} & a_3(A-1) & a_2(A-1) \\
		-2\alpha & -2\alpha & 1-2a_3\alpha & -2a_2\alpha \\
		-2\beta & -2\beta & -2a_3\beta & 1-2a_2\beta \
	\end{bmatrix}, 
	\end{equation}
	where 
	$\alpha = a_2^2/(4(a_0a_2-1))$, 
	$\beta = a_3^2/(4(a_0a_3-1))$ and 
	$A = a_3\alpha + a_2\beta$. 
\end{thm}

\begin{proof}
	We start from the $2A$-octanomial normal form  (\ref{eq:2A-octa}). 
	We will make coordinate change and add some conditions of parameters 
	to reduce the equation to the form 
	\begin{equation}\label{eq:4A_1}
		x_0^2x_1 + x_1^2g_1(x_2,x_3) + g_3(x_2,x_3) = 0, 
	\end{equation}
	where $g_j(x_2,x_3)$ are homogeneous polynomials in $x_2,x_3$ of degree $j$. 
	In fact, 
	a surfaces defined by (\ref{eq:4A_1}) admits the automorphism 
	$(x_0,x_1,x_2,x_3) \mapsto (ix_0,-x_1,x_2,x_3)$ of type $4A$. 
	
	By applying coordinate change 
	$(x_0, x_1) \mapsto (x_0 + x_1, -x_0 + x_1)$, 
	the equation is the form 
	\[ (x_1^2 - x_0^2)(2x_1 + a_3x_2 + a_2x_3) + x_2x_3(2a_0x_1 + x_2 + x_3) = 0, \]
	whose automorphism of type $2A$ is $x_0 \mapsto -x_0$. 
	By the change $x_1 \mapsto (1/2)(x_1 - a_3x_2 -a_2x_3)$, 
	we reduce the form 
	\begin{equation}\label{eq:4A_2}
	-x_0^2x_1 + \frac{1}{4}(x_1 - a_3x_2 - a_2x_3)^2x_1 
	+ x_2x_3(a_0x_1 + (1-a_0a_3)x_2 + (1-a_0a_2)x_3) = 0.
	\end{equation}
	
	By setting $x_0 = 0$ in (\ref{eq:4A_2}),  
	we have an elliptic curve $C$. 
	We will make coordinate change in $x_2,x_3$ to reduce the equation of $C$ to the form
	$x_1^2g_1(x_2,x_3) + g_3(x_2,x_3) = 0$. 
	By the coordinate change $(x_2, x_3)  \mapsto (\alpha x_1 + x_2, \beta x_1 + x_3)$, 
 	the terms $x_1x_2^2$, $x_1x_3^2$ vanish because of the definitions of $\alpha, \beta$, 
	and so does the ones  $x_1^3$, $x_1x_2x_3$ by the condition $(\ast)$. 
	 
	As a result, 
	the surface is defined by the equation 
	\begin{equation}\label{eq:4A_3}
		-x_0^2x_1 
		+ x_1^2( 
			\gamma x_2
			+ \delta x_3 )
		+(1-a_0a_3)x_2^2x_3 + (1-a_0a_2)x_2x_3^2 = 0, 
	\end{equation}
	where 
	$\gamma = \beta^2(1-a_0a_2) + 2 \alpha \beta (1 - a_0a_3) + a_0 \beta - a_3(1 - A)/2$ and 
	$\delta = \alpha^2 (1-a_0a_3) + 2 \alpha \beta (1 - a_0a_2) + a_0 \alpha - a_2(1 - A)/2$. 
\end{proof}

\subsection{The stratum $4B$ in $p \neq 2$}
The stratum $4B$ is of dimension $1$ in $\mathcal{M}(k)$. 
The action of an automorphism of type $4B$ on $\mathbb{P}^3$ is given by
\[ \begin{bmatrix}
	0 & 0 & 0 & 1 \\
	1 & 0 & 0 & 0 \\
	0 & 1 & 0 & 0 \\
	0 & 0 & 1 & 0 
\end{bmatrix}, \]
up to projective equivalence \cite[Lemma 11.1]{DD19}. 
An automorphism of type $4B$ is distinguished from one of type $4A$ by its square which is of type $2B$. 

Let $g$ be an automorphism of  a cubic surface of type $4B$. 
The $g$-orbits of the $27$ lines are the following: 
$(E_1, E_5, E_2, F_{34})$, $(G_6, F_{25}, F_{12}, F_{15})$, $(F_{16}, G_2, G_1, F_{26})$, 
$(E_4, F_{14}, F_{35}, F_{24})$, $(E_3, F_{13}, F_{45}, F_{23})$, 
$(F_{36}, F_{36})$, $(G_3, G_4)$, $(G_5, F_{56}), (E_6)$ in a suitable labeling.

An automorphism of a cubic surface of type $4B$ is determined by the arrangement of Eckardt points. 
Let $X$ be a cubic surface which admits an automorphism $g$ of type $4B$, 
then there is a $g$-invariant tritangent plane $\pi$. 
In $p \neq 2$, 
$\pi$ does not correspond to an Eckardt point, and 
there are exactly $6$ Eckardt points on $\pi$  
whose corresponding automorphisms of type $2A$ generate 
a group isomorphic to $\mathfrak{S}_4$ containing $g$ \cite[Lemma 11.5]{DD19}.  
In $p=2$, there are more Eckardt points on $\pi$ we will see later (Subsection \ref{sec:6E=4B=4A_char2}). 

Although an automorphism of type $4B$ is realized by the arrangement of Eckardt points, 
we can get a $4B$-octanomial normal form directly. 

\begin{thm}[The $4B$-octanomial normal form in $p \neq 2$] \label{4B-octa}
	In $p \neq 2$, 
	a general cubic surface admitting an automorphism of type $4B$ is 
	projectively isomorphic to the surface given by the octanomial form with $a_0 = a_1 = a_2 = a_3$: 
	\begin{equation}\label{eq:4B-octa}
		x_0x_1(x_0 + x_1 + a_0x_2 + a_0x_3) + x_2x_3(a_0x_0 + a_0x_1 + x_2 + x_3) = 0.
	\end{equation}
	The surfaces given by the above have the automorphism of type $4B$:  
	\begin{equation}\label{eq:4B-auto}
	\begin{bmatrix}
		0 & 0 & 0 & 1\\
		0 & 0 & 1 & 0 \\
		1 & 0 & 0 & 0 \\
		0 & 1 & 0 & 0 
	\end{bmatrix}. 
	\end{equation}
	The invariant tritangent plane is given by $x_0 + x_1 + x_2 + x_3 = 0$. 
\end{thm}

\begin{proof}
	Since (\ref{eq:4B-auto}) is an automorphism of order $4$ and 
	its square is of type $2B$, 
	it is of type $4B$. 
\end{proof}

\begin{rmk}\label{3D_to_4B}
	The stratum $4B$ is the specialization of the strata $2B$ and $3D$. 
	The $4B$-octanomial normal form preserves 
	the specialization of parameters from the $2B$-octanomial normal form (\ref{eq:2B-octa}), 
	but does not from the $3D$-octanomial normal form (\ref{eq:3D-octa}).  
	Choosing the octanomial parameter which preserves the both specialization, 
	we have $a_0 = a_1 = 0$ and $a_2 = a_3$, 
	which is the same as the $6E$-octanomial normal form (\ref{eq:6E-octa}). 
	
	However, 
	we have another $4B$-octanomial normal form 
	which preserves the specialization $3D \to 4B$ but does not $2B \to 4B$:
	\textit{a general cubic surface admitting an automorphism of type $4B$ is 
	projectively isomorphic to the surface given by the octanomial form with 
	$a_0 =a_1 = 0$, $a_2 + a_3 + 2c = 0$, 
	where $c$ satisfies the equation $c(a_2 + c)(a_3 + c) = 1$: 
	\begin{equation}\label{eq:4B'-octa}
		x_0x_1(x_0 + x_1 + a_3x_2 + a_2x_3) + x_2x_3(x_2 + x_3) = 0. 
	\end{equation}}
	The surface defined by the above has the trihedral line $\ell \colon x_2 = x_3 = 0$ 
	which is the same as the $3D$-octanomial normal form. 
	The tritangent plane $x_2 + x_3 = 0$ contains $\ell$ 
	and the Eckardt point $q_0=(a_2 + c, -(a_3 + c), 1,-1)$ 
	which is not on $\ell$. 	
	(The corresponding tritangent plane to $q_0$ is defined by $x_0 + x_1 + (a_3 + c)x_2 + (a_2 + c)x_3 = 0$ 
	and its tritangent trio is given by the intersection with the planes 
	$x_2 + x_3 =0$, $c(a_2 + c)x_1 + x_2 = 0$ and $c(a_3 + c)x_1 + x_3 = 0$. )
	The product of ones of type $2A$ corresponding to 
	the $2$ Eckardt points on $\ell$ and $q_0$ 
	give the one of type $4B$. 
\end{rmk}

\subsection{The stratum $6E$}
The stratum $6E$ is of dimension $1$ in $\mathcal{M}(k)$.
An automorphism of type $6E$ is realized as a product of 
commuting automorphisms of type $3D$ and $2A$. 

\begin{lem}[{\cite[Lemma 12.8]{DD19}}]\label{DD19:12.8}
	Let $X$ be a cubic surface which admits an automorphism of type $6E$. 
	Then there exist $4$ Eckardt points on $X$, 
	of which $3$ lie on a trihedral line contained in the tritangent plane of the fourth. 
	Their corresponding automorphisms of type $2A$ generate 
	a group isomorphic to $\mathfrak{S}_3 \times \mathbb{Z}/2\mathbb{Z}$. 
\end{lem}

\begin{thm}[The $6E$-octanomial normal form]\label{6E-octa}
	Every cubic surface admitting an automorphism of type $6E$ is 
	projectively isomorphic to the surface given by the octanomial form with $a_0 = a_1 = 0, \ a_2 = a_3$: 
	\begin{equation}\label{eq:6E-octa}
		 x_0x_1(x_0 + x_1 + a_2x_2 + a_2x_3) + x_2x_3(x_2 + x_3) = 0.
	\end{equation}
	The surfaces given by the above have the automorphism of type $6E$: 
	\begin{equation}\label{eq:6E-auto}
	\begin{bmatrix}
		0 & 1 & 0 & 0 \\
		-1 & -1 & -a_2 & -a_2 \\
		0 & 0 & 0 & 1 \\
		0 & 0 & 1 & 0 
	\end{bmatrix}. 
	\end{equation}

\end{thm}

\begin{proof}
	The surface defined by (\ref{eq:6E-octa}) has the automorphism (\ref{eq:3D-auto}) of type $3D$ 
	and the one of type $2A$ exchanging $x_2$ with $x_3$, 
	which are commuting. 
	The product of them (\ref{eq:6E-auto}) is the one of type $6E$. 
	
	Conversely, 
	let $X$ be a cubic surface admitting an automorphism of type $6E$. 
	By Lemma \ref{DD19:12.8}, 
	there is a trihedral line $\ell$, and is a tritangent plane containing $\ell$ and corresponding to an Eckardt point. 
	There is a conjugate pair of triads of tritangent trios $(\{\pi_1,\pi_2,\pi_3\}, \{\pi'_1,\pi'_2,\pi'_3\})$ such that 
	each of $\pi_i \ ( i=1,2,3)$ corresponds to an Eckardt point and 
	$\ell = \pi'_1\cap \pi'_2 \cap \pi'_3$ by Lemma \ref{DD19:10.1}. 
	Since tritangent planes which contain $\ell$ are $\pi'_i \ (i=1,2,3)$ only, 
	we can assume that $\pi'_1$ corresponds to an Eckardt point. 
	Then $X$ is defined by $\pi_1\pi_2\pi_3 + \pi'_1\pi'_2\pi'_3 = 0$. 
	In the same way as the proof of Theorem \ref{2A-octa} and Theorem \ref{3D-octa}, 
	we get the octanomial form with $a_2 = a_3$ and $a_0 = a_1 = 0$. 
\end{proof}

\subsection{The stratum $6E = 4A = 4B $ in $p=2$}	\label{sec:6E=4B=4A_char2}
Let $p=2$. 
Since the strata $6E, 4A, 4B$ coincide \cite[Lemma 12.9]{DD19}, 
we give automorphisms of type $4A$ and $4B$ of the $6E$-octanomial normal form (\ref{eq:6E-octa}). 

\begin{thm}[The $6E = 4B = 4A$-octanomial normal form in $p=2$]\label{6E=4B=4A-octa_char2}
	In $p=2$, 
	every cubic surface admitting an automorphism of type $6E$, $4A$, or $4B$ 
	is projectively isomorphic to the surface given by the $6E$-octanomial normal form (\ref{eq:6E-octa}): 
	\[  x_0x_1(x_0 + x_1 + a_2x_2 + a_2x_3) + x_2x_3(x_2 + x_3) = 0. \]
	The surfaces given by the above have the automorphisms of all types $6E$, $4A$ and $4B$: 
	\begin{equation}\label{eq:6E_4B_4A-auto}
	\begin{bmatrix}
		0 & 1 & 0 & 0 \\
		1 & 1 & a_2 & a_2 \\
		0 & 0 & 0 & 1 \\
		0 & 0 & 1 & 0 
	\end{bmatrix}, \
		\begin{bmatrix}
		1 & 0 & \mu_2^2 & \mu_2^2 \\
		0 & 1 & \mu_1^2 & \mu_1^2 \\
		\mu_	1 & \mu_2 & \mu_2^2\mu_1 & 1 + \mu_2^2\mu_1 \\
		\mu_1 & \mu_2 & 1 + \mu_2^2\mu_1 & \mu_2^2\mu_1
	\end{bmatrix}, \text{and} \ 
	\begin{bmatrix}
		1 & 0 & \mu_1^2 & \mu_1^2 \\
		1 & 1 & a_2 & a_2 \\
		\mu_1 & \mu_1 & 1 + \mu_1^3 & \mu_1^3 \\
		\mu_1 & \mu _1& \mu_1^3 & 1 + \mu_1^3
	\end{bmatrix}, 
	\end{equation}
	respectively. 
	Here, $\lambda, \mu$ are the distinct roots of the equation $t^3 + a_2t + 1 = 0$ of $t$. 
\end{thm}

\begin{proof}
	The automorphism of type $6E$ is the one we see in Theorem \ref{6E-octa}. 
	We will observe the arrangement of the Eckardt points on the surface 
	to obtain automorphisms of type $4A$ and $4B$. 
	
	The cubic surface defined by (\ref{eq:6E-octa}) has the tritangent plane $\pi : x_2 + x_3 = 0$ 
	which corresponds to an Eckardt point $c=(0,0,1,1)$ and contains exactly $13$ Eckardt points. 
	The $3$ lines on $\pi$ are the intersection with the planes 
	$x_0 = 0$, $x_1 = 0$,and $x_0 + x_1 + a_2x_2 + a_2x_3 = 0$, 
	written by $\ell_0, \ell_1, \ell_2$ respectively. 
	Let $\mu_i$ ($i=1,2,3$) be distinct roots of the equation $t^3 + a_2t + 1 = 0$. 
	Each line contains the Eckardt point $c$, 
	and the other $4$ ones given by the following\footnote{
		We can determine the Eckardt point $t_{\mu_i}$ 
		by finding the tritangent plane containing the line $\ell_2$.  
		The tritangent plane corresponding to $t_{\mu_i}$ is 
		$\mu_i(x_0 + x_1) + x_2 + x_3 = 0$, 
		and its tritangent trio is given by the intersection with the planes 
		$x_0 + x_1 = 0$, $\mu_i x_0 + x_2 = 0$ and $\mu_i x_1 + x_2 = 0$.
		Applying automorphism of type $2A$ corresponding to $r_0, s_0$, 
		we get the other Eckardt points $r_{\mu_i}, s_{\mu_i}$. 
		}: 
	$\ell_0$ contains $r_0 = (0,1,0,0)$ and  $r_{\mu_i} = (0,1,\mu_i,\mu_i)$;  
	$\ell_1$ contains $s_0 = (1,0,0,0)$ and $s_\mu = (1,0,\mu_i,\mu_i)$; 
	$\ell_2$ contains $t_0 = (1,1,0,0)$ and $t_\mu = (1,1,\mu_i,\mu_i)$.  
	The corresponding automorphisms of type $2A$ are given by 
	\[ 
	r_0 = 
	\begin{bmatrix}
		1 & 0 & 0 & 0\\
		1 & 1 & a_2 & a_2 \\
		0 & 0 & 1 & 0 \\
		0 & 0 & 0 & 1 
	\end{bmatrix}, 
	s_0 = 
	\begin{bmatrix}
		1 & 1 & a_2 & a_2 \\
		0 & 1 & 0 & 0 \\
		0 & 0 & 1 & 0 \\
		0 & 0 & 0 & 1
	\end{bmatrix}, 
	t_0 = 
	\begin{bmatrix}
		0 & 1 & 0 & 0 \\
		1 & 0 & 0 & 0 \\
		0 & 0 & 1 & 0 \\
		0 & 0 & 0 & 1 
	\end{bmatrix}, 
	t_{\mu_i} = 
	\begin{bmatrix}
		0 & 1 & \mu_i^2 & \mu_i^2 \\
		1 & 0 & \mu_i^2 & \mu_i^2 \\
		\mu_i & \mu_i & \mu_i^3 + 1 & \mu_i^3 \\
		\mu_i & \mu_i & \mu_i^3 & \mu_i^3 + 1
	\end{bmatrix}. 
	\]
	Then, one of automorphisms of type $4A$ and $4B$ are realized by 
	$t_{\mu_1} s_0 t_{\mu_2} r_0$ and $s_0 r_0 t_{\mu_1}$ respectively. 
	
	Conversely, 
	Let $X$ be a cubic surface which admits an automorphism of type $6E$, $4A$, or $4B$, 
	then $X$ admits ones of all types. 
	It follows from Theorem \ref{6E-octa} that $X$ is defined by the equation (\ref{eq:6E-octa}) in a suitable coordinates. 
\end{proof}

\begin{rmk}
	In the above, 
	the plane $\pi$ is  the \textit{canonical plane} and the point $c$ is the \textit{canonical point} of the surface, 
	called in \cite[Subsection 4.1]{DD19}. 
	In the case, the canonical plane is a tritangent plane with the canonical point being its Eckardt point. 
\end{rmk}

\subsection{The stratum $3C$} \label{subsec:3C}
The stratum $3C$ is of dimension $0$ in $\mathcal{M}(k)$, 
corresponding to the isomorphism class of the Fermat cubic surface. 
The action of an automorphism of type $3C$ on $\mathbb{P}^3$ is given by 
$\diag (\zeta_3, \zeta_3, 1, 1)$ in $p\neq 3$ up to projective equivalence \cite[Theorem 10.4]{DD19}. 
In $p=3$, 
the Fermat cubic surface is non-reduced, 
and there is no cubic surface which admits an automorphism of type $3C$. 
So let $p\neq 3$. 

Let $g$ be an automorphism of a cubic surface of type $3C$. 
In a suitable labeling, 
the $g$-orbits of the $27$ lines are given by: 
$(F_{16}, F_{26}, F_{36})$, $(F_{15}, F_{25}, F_{35})$, $(F_{14}, F_{24}, F_{34})$, 
$(E_1, E_2, E_3)$, $(F_{12}, F_{23}, F_{13})$, $(G_1, G_2, G_3)$, 
$(E_4), (E_5), (E_6)$, $(G_4), (G_5), (G_6)$, $(F_{45}), (F_{56}),(F_{46})$, 
preserving the $3$ conjugate pairs of triads of tritangent trios (\ref{eq:ord3}). 
There are $9$ invariant lines and $6$ skew triples. 

Every cubic surface which admits an automorphism of type $3C$ is 
projectively isomorphic to the Fermat cubic surface: $x_0^3 + x_1^3 + x_2^3 + x_3^3 = 0$ \cite[Lem 10.14]{DD19}. 
The Fermat cubic surface has $18$ Eckardt points in $p\neq 2$, 
and $45$ Eckardt points in $p=2$. 
Each tritangent plane defined by $x_i + \zeta_3^u x_j = 0$ ($0\leq i<j \leq 3$, $u = 0,1,2$) corresponds to an Eckardt point in any characteristic.
There are $\binom{4}{2}\times 3 = 18$ such tritangent planes.  
The other tritangent planes are defined by 
$\zeta_3^{u_0} x_0 + \zeta_3^{u_1} x_1 + \zeta_3 ^{u_2} x_2 + \zeta_3^{u_3} x_3 = 0$ ($u_i = 0,1,2$),
which do not correspond to Eckardt points in $p\neq 2$ but do in $p=2$. 
The automorphism group of the surface is isomorphic to $\mathbb{Z}/3\mathbb{Z} \rtimes \mathfrak{S}_4$ of order $648$ in $p\neq 2$, $\PSU_4(2)$ of order $25920$ in $p=2$ \cite[Lemma 5.1]{DD19}.  

Although the  automorphisms of type $3C$ are realized as products of ones of type $2A$, 
we directly make coordinate change to reduce an octanomial form with no parameter.  

\begin{thm}[The $3C$-octanomial normal form]\label{3C-octa}
	Every cubic surface admitting an automorphism of type $3C$ is 
	projectively isomorphic to the surface defined by the octanomial form with $a_0 = a_1 = a_2 = a_3 = 0$: 
	\begin{equation}\label{eq:3C-octa}
		 x_0x_1(x_0+x_1) + x_2x_3(x_2+x_3) = 0. 
	 \end{equation}
	 The surface given by the above has the automorphism of type $3C$: 
	\begin{equation}\label{eq:3C-auto}
	\begin{bmatrix}
		\zeta_3 & 0 & 0 & 0 \\
		0 & \zeta_3 & 0 & 0 \\
		0 & 0 & 1 & 0 \\
		0 & 0 & 0 & 1 
	\end{bmatrix}. 
	\end{equation}
\end{thm}
\begin{proof}
	By applying the coordinate change 
	$(x_0, x_1, x_2, x_3) \mapsto 
	(x_0 + \zeta_3 x_1, \zeta_3x_0 + x_1 , x_2 + \zeta_3x_3, \zeta_3x_2 + x_3)$, 
	the surface (\ref{eq:3C-octa}) is defined by the Fermat cubic form. 
\end{proof}

\begin{rmk}[All octanomial parameters of the Fermat cubic surface]\label{octa_para_3C}
	Let $p\neq 3$ and $X$ the Fermat cubic surface. 
	We directly compute all octanomial parameters $(a_0,\dots ,a_3)$ of $X$. 
	We will apply the coordinate change as Remark \ref{octa_para} to each Cayley--Salmon equation given by the following:
	\begin{enumerate}
		\item[(i)]
		$(x_0 + x_1)(x_0 + \zeta_3x_1)(x_0 + \zeta_3^2x_1) 
		+ (x_2 + x_3)(x_2 + \zeta_3x_3)(x_2 + \zeta_3^2x_3) = 0$; 
		\vspace{3pt}
		\item[(ii)] 
		$(x_0 + x_1 +  x_2 + x_3)
		(x_0 +  x_1 +  x_2 + \zeta_3x_3)
		(x_0 +  x_1 +  x_2 + \zeta_3^2x_3)$ \\
		$-3 (x_0 + x_1)( x_1 + x_2)(x_2 + x_0) = 0$; 
		\vspace{3pt}
		\item[(iii)]
		$(x_0 + x_1)
		(\zeta_3x_0 + \zeta_3^{2}x_1 + x_2 +  x_3)			
		(\zeta_3^2 x_0 + \zeta_3 x_1 +  x_2 +  x_3)$ \\
		$+ ( x_2 +  x_3)
		(x_0 +  x_1 + \zeta_3 x_2 + \zeta_3^{2} x_3)
		(x_0 + x_1 + \zeta_3^{2} x_2 + \zeta_3 x_3) = 0$. 
	\end{enumerate}
	
	There are $3=\binom{4}{2}/2$ Cayley--Salmon equations of the form (i) 
	counted by permuting the coordinates, 
	which give the octanomial parameter $(0,0,0,0)$ as in Theorem \ref{3C-octa}. 
	The corresponding conjugate pair of triads of tritangent trios 
	consists of invariant lines of an automorphism of type $3C$, 
	and all tritangent planes correspond to Eckardt points.  
	
	The number of the Cayley--Salmon equation of the form (ii) is $36 = 4 \times 3^2$ 
	counted by scaling $\zeta_3^i$ with the coordinates in addition to the permuting. 
	They give the $12= 3 \times 2^2$ octanomial parameters 
	$(-\zeta_3^{i+1} \cdot 2\cdot 3^{1/3}/(1-\zeta_3), \zeta_3^i \cdot 2\cdot 3^{1/3}/(1-\zeta_3), 0, 0)$, 
	counted ones given by exchanging $a_0, a_1$ and by  $a_2 , a_3$ ($i = 0,1,2$). 
	
	There are  $81 = \binom{4}{2}/2 \times 3^3$ Cayley--Salmon equations of the form (iii) counted in the same way. 
	They give the following $27=4 + 8 + 12 + 1 + 2$ octanomial parameters: 
	$(0,-2,0,-2)$;
	$(0,-2\zeta_3,0,-2\zeta_3^2)$; 
	$(-2\zeta_3^{2i}/3, -2\zeta_3^{2i}/3, 0, -2\cdot 3^{1/3}\zeta_3^i/3)$;
	$(2,2,2,2)$; 
	$(2\zeta_3,2\zeta_3, 2\zeta_3^2, 2\zeta_3^2)$, 
	counted in the same way as the above. 
	
	As a result, 
	there are $40 \ (= 25920/648)$ octanomial parameters in $p\neq 2$. 
	In $p=2$, 
	the all octanomial parameters are $(0,0,0,0)$, being consistent with Remark \ref{octa_para}. 
\end{rmk}

\subsection{The stratum $5A$}\label{subsec:5A}
The stratum $5A$ is of dimension $0$ in $\mathcal{M}(k)$, 
corresponding to the isomorphism class of the Clebsch cubic surface. 
The orbits of the $27$ lines by one of type $5A$ is given by: 
$(E_2)$, $(G_2)$, 
$(E_1,E_5,E_6,E_4,E_3)$, $(G_1,G_5,G_6,G_4,G_3)$, 
$(F_{12},F_{25},F_{26},F_{24},F_{23})$, 
$(F_{13},F_{15},F_{56},F_{46},F_{34})$, 
$(F_{14},F_{35},F_{16},F_{45},F_{36})$,
in a suitable labeling. 

Every cubic surface which admits an automorphism of type $5A$ is 
projectively isomorphic to the Clebsch cubic surface: 
\begin{equation}\label{eq:clebsch}
	\sum_{i=0}^4 x_i = \sum_{0\leq i<j<k \leq 4} x_ix_jx_k = 0, 
\end{equation}
defined in $\mathbb{P}^4$ \cite[Thm 12.2]{DD19}. 
In $p = 5$, 
the Clebsch cubic surface has a singular point, 
and there is no cubic surface which admits an automorphism of type $5A$ \cite[Lem 12.1]{DD19}. 
So let $p\neq 5$. 
We directly make coordinate change and get the octanomial form of the Clebsch cubic surface. 

\begin{thm}[The $5A$-octanomial normal form]\label{5A-octa}
	Every cubic surface admitting an automorphism of type $5A$ is 
	projectively isomorphic to the surface defined by the octanomial form with $a_0 = a_1 =0$,  $a_2 = a_3 = -2$: 
	\begin{equation}\label{eq:5A-octa}
		x_0x_1(x_0+x_1 - 2 x_2 - 2 x_3) + x_2x_3(x_2+x_3) = 0.
	\end{equation}
	The surface given by the above has the automorphism of type $5A$: 
	\begin{equation}\label{eq:5A-auto}
	\begin{bmatrix} 
		-1 & 0 & 0 & 1\\
		-1 & -1 & 2 & 2 \\
		0 & 0 & 0 & 1 \\
		-1 & -1 & 1 & 1 
	\end{bmatrix}.
	\end{equation}
\end{thm}
\begin{proof}
	The Clebsch cubic surface is defined by the Cayley--Salmon equation: 
	\begin{equation}\label{eq:clebsch_cs}
		(x_1 + x_2 + x_3)(x_0 + x_2 + x_3)(x_0 + x_1) + 
		x_2x_3(x_2 + x_3) = 0. 
	\end{equation}
	By making the coordinate change 
	$(x_0, x_1) \mapsto (x_0 -x_2 -x_3, x_1, x_1 - x_2 - x_3)$, 
	the surface (\ref{eq:clebsch_cs}) is defined by the octanomial form (\ref{eq:5A-octa}). 
\end{proof}

\begin{rmk}\label{4B_to_5A}
	The stratum $5A$ is the specialization of the strata $4B$ and $6E$. 
	The $5A$-octanomial normal form preserves the specialization of parameters from 
	the $6E$-octanomial normal form (\ref{eq:6E-octa}), 
	but does not from the $4B$-octanomial normal form (\ref{eq:4B-octa}). 
	
	However, 
	we have another $5A$-octanomial form 
	which is the specialization of another $4B$-octanomial form we give in Remark \ref{3D_to_4B}. 
	By applying the coordinate change $x_2 \mapsto -(x_2 + x_3)$ to (\ref{eq:5A-octa}), 
	we have the octanomial form with $a_0 = a_1 = a_2 = 0$, $a_3 = 2$: 
	\begin{equation}
		x_0x_1 ( x_0 + x_1 + 2 x_2) + x_2 x_3(x_2 + x_3) = 0, 
	\end{equation}
	which is the specialization of another $4B$-octanomial normal form with $c=-1$. 
\end{rmk}

\subsection{The stratum $12A$ in $p\neq2,3$}	
The stratum $12A$ is of dimension $0$ in $\mathcal{M}(k)$. 
Let $X$ be a cubic surface admitting an automorphism $g$ of type $3A$. 
When the elliptic curve on the axis of $g$ has an automorphism $h'$ of order $4$, 
a suitable extension $h$ of $h'$ to an automorphism of $X$ is commuting with $g$, 
and the product $gh$ is the one of type $12A$. 
We will give the condition on the parameters of 
the $3A$-octanomial normal form (\ref{eq:3A-octa}), 
which admits the automorphism of type $4A$ that is commuting with the one of type $3A$. 

\begin{thm}[The $12A$-octanomial normal form]\label{12A-octa}
	In $p \neq 2,3$, 
	every cubic surface admitting an automorphism of type $12A$ is 
	projectively isomorphic to the surface defined by the octanomial form satisfying the condition $(\star)$: 
	\begin{equation}\label{eq:12A-octa}
		 x_0x_1(x_0 + x_1 + a_3 x_2  - \zeta_3 a_3 x_3) + x_2x_3(x_2 + x_3) = 0.
	\end{equation}
	$(\star)$: \ 
	$a_0 = a_1 = 0$, 
	$a_2 + \zeta_3a_3 = 0$ and 
	$a_3^6 + 4\zeta_3(1-\zeta_3) a_3^3 -8 = 0$. 

	The surface given by the above has the automorphism of type $12A$: 
	\begin{equation}\label{eq:12A-auto}
		\begin{bmatrix}
			1 & 0 & 0 & 0 \\
			\frac{a_3'^3 - 6 + 6 i}{12} & i 
			& \frac{a_3'(1 + i)}{2(1 - \zeta_3^2)} & \frac{a_3'(1 + i)(1 + \zeta_3^2)}{2(1-\zeta_3^2)} \\
			-\frac{a_3'^2}{6} & 0 & 0 & -\zeta_3^2 \\
			-\frac{\zeta_3 a_3'^2}{6} & 0 & \zeta_3^2 & \zeta_3^2 
		\end{bmatrix}, 
	\end{equation}
	where $a_3'=(1-\zeta_3^2)a_3$. 
\end{thm}

\begin{proof}
	We start from the $3A$-octanomial normal form (\ref{eq:3A-octa}). 
	By the coordinate change $(x_2, x_3 ) \mapsto (x_2 + \zeta_3x_3, \zeta_3 x_2 + x_3)$, 
	we have 
	\begin{equation}\label{eq:12A_1}
		-x_3^3 -x_2^3 + a_3' x_0x_1x_2 + x_0x_1(x_0 + x_1) = 0. 
	\end{equation}
	This surface admits the automorphism $x_3 \mapsto \zeta_3 x_3$ of type $3A$, 
	whose axis is defined by $x_3 = 0$. 

	Let $C$ be the intersection of the surface (\ref{eq:12A_1}) and the axis $x_3 = 0$, 
	which is an elliptic curve. 
	By the coordinate change  
	$x_1 \mapsto x_1 - (1/2)(x_0 + a_3'x_2)$, 
	the term $x_0^2x_1$ vanishes. 
	Then by 
	$x_2 \mapsto - (a_3'^2 / 12) x_0 + x_2$, 
	$C$ is defined by 
	\begin{equation}\label{eq:12A_2}
		 x_0 x_1^2 - x_2^3 - \frac{a_3'^6 - 36a_3'^3 + 216}{864} x_0^3 + \frac{a_3'(a_3'^3 - 24)}{48} x_0^2 x_2 = 0, 
	\end{equation}
	which is a Weierstrass form. 
	Since the term $x_0^3$ vanishes, 
	$C$ admits the automorphism $\diag (1,i,-1)$ of order $4$. 
	
	The cubic surface is now defined by the equation 
	\begin{equation}\label{eq:12A_3}
		-x_3^3 + x_0 x_1^2 - x_2^3 + \frac{a_3'(a_3'^3 - 24)}{48} x_0^2 x_2 = 0, 
	\end{equation}
	which has the automorphism 
	$\diag (1,i,-1,-1)$ of type $4A$. 
	The product $\diag (1,i,-1,-\zeta_3)$ with the one of type $3A$ 
	 is the automorphism of type $12A$. 

\end{proof}

\begin{rmk}
	In $p=2$ (resp. $p=3$), 
	the stratum $12A$ is the same as the one $3C$ (resp. $8A$). 
	We will give an automorphism of the $3C$ (resp. $8A$)-octanomial normal form of type $12A$ 
	in Subsection \ref{sec:3C=5A=12A_char2} (resp.  \ref{sec:8A=12A_char3}). 
\end{rmk}

\subsection{The stratum $3C = 5A = 12A$ in $p=2$}\label{sec:3C=5A=12A_char2}
Let $p=2$. 
The strata $3C$, $5A$ and $12A$ coincide \cite[Lemma 12.14]{DD19}. 
Since the $5A$-octanomial normal form (\ref{eq:5A-octa}) is the same to  
the $3C$-octanomial normal form (\ref{eq:3C-octa}), 
we give automorphisms of type $12A$ of the surface. 

\begin{thm}[The $3C=5A=12A$-octanomial normal form in $p=2$] \label{3C=5A=12A-octa_char2}
	In $p=2$, 
	every cubic surface admitting an automorphism of type $3C$, $5A$ or $12A$ is 
	projectively isomorphic to the surface given by the $3C$-octanomial normal form (\ref{eq:3C-octa}): 
	\[  x_0x_1(x_0+x_1) + x_2x_3(x_2+x_3) = 0. \]
	The surface given by the above has the automorphisms of all type $3C$, $5A$ and $12A$: 
	\begin{equation}\label{eq:3C=5A=12A-auto_char2}
		\begin{bmatrix}
			\zeta_3 & 0 & 0 & 0 \\
			0 & \zeta_3 & 0 & 0 \\
			0 & 0 & 1 & 0 \\
			0 & 0 & 0 & 1
		\end{bmatrix}, \
		\begin{bmatrix}
			1 & 0 & 0 & 1 \\
			1 & 1 & 0 & 0 \\
			0 & 0 & 0 & 1 \\
			1 & 1 & 1 & 1
		\end{bmatrix}, \text{and} \
		\begin{bmatrix}
			\zeta_3^2 & \zeta_3^2 & 0 & \zeta_3 \\
			\zeta_3^2 & 0 & 0 & 1 \\
			1 & \zeta_3^2 & 1 & \zeta_3 \\
			0 & 0 & 0 & 1
		\end{bmatrix}. 
	\end{equation}
\end{thm}

\begin{proof}
	The automorphisms of type $3C$ and $5A$  of (\ref{eq:3C=5A=12A-auto_char2}) are the ones 
	we see in Theorem \ref{3C-octa} and Theorem \ref{5A-octa}. 
	We will give the automorphism of type $12A$ of the surface (\ref{eq:3C-octa}). 
	By the coordinate change $(x_0, x_1) \mapsto (x_0 + \zeta_3x_1, \zeta_3x_0 + x_1)$, 
	the surface (\ref{eq:3C-octa}) is defined by the equation 
	\begin{equation}\label{eq:3C=5A=12A_char2_1}
		x_0^3 + x_1^3 + x_2x_3(x_2 + x_3) = 0, 
	\end{equation}
	which admits the automorphism $\diag (\zeta_3,1,1,1)$ of type $3A$. 
	The elliptic curve on the axis $x_0 = 0$ 
	has the automorphism $(x_1, x_2, x_3) \mapsto (x_1 + x_3, x_1 + x_2 + \zeta_3 x_3, x_3)$ of order $4$. 
	Then the surface (\ref{eq:3C=5A=12A_char2_1}) admits the extended automorphism by fixing $x_0$, 
	which is of type $4A$. 
	The product with the one of type $3A$ is the automorphism of type $12A$: 
	\[ 
	\begin{bmatrix}
		\zeta_3 & 0 & 0 & 0 \\
		0 & 1 & 0 & 1 \\
		0 & 1 & 1 & \zeta_3 \\
		0 & 0 & 0 & 1
	\end{bmatrix}.
	\]
\end{proof}

\subsection{The stratum $8A$}
The stratum $8A$ is of dimension $0$ in $\mathcal{M}(k)$. 
In $p\neq 2,3$, the corresponding surface admits only one Eckardt point, which has little information as in the case of the stratum $4A$. 
In $p=2$, 
there is no cubic surface which admits an automorphism of type $8A$ \cite[Lem 12.11]{DD19}. 
So let $p\neq 2$. 

Let $X$ be a cubic surface admitting an automorphism $g$ of type $4A$. 
When the elliptic curve on the axis of $g^2$ has an automorphism $h'$ of order $4$, 
a suitable extension $h$ of $h'$ to an automorphism of $X$ is commuting with $g^2$, 
and the product $g^2h$ is the one of type $8A$. 
We will give the condition for the parameters of the $4A$-octanomial normal form (\ref{eq:4A-octa}) 
which admits an automorphism of type $8A$. 

\begin{thm}[The $8A$-octanomial normal form]\label{8A-octa}
	Every cubic surface admitting an automorphism of type $8A$ is 
	projectively isomorphic to the surface 
	defined by the octanomial form with $a_0,a_2,a_3$ satisfying the condition $(\dag)$: 
	\begin{equation}\label{eq:8A-octa}
		x_0x_1(x_0 + x_1 + a_3 x_2 + a_2 x_3 ) + x_2x_3(a_0 x_0 + a_0x_1 + x_2 + x_3) = 0.
	\end{equation}
	$(\dag)$: 
	Satisfying the condition $(\ast)$ of Theorem \ref{4A-octa}, 
	and $\gamma (1-a_0a_2) - 2 \delta (1-a_0a_3) = 0$. 
	
	The surface given by the above has the automorphism of type $8A$: 
	\begin{equation}\label{eq:8A-auto}
		\begin{bmatrix}
		S + \frac{\zeta_8}{2} & 
		S - \frac{\zeta_8}{2} & 
		a_3(S + \frac{1}{2}) & 
		a_2(S -\frac{1}{2}) + \frac{a_3\delta}{\gamma} \\
		S - \frac{\zeta_8}{2} & 
		S + \frac{\zeta_8}{2} & 
		a_3(S + \frac{1}{2}) & 
		a_2(S - \frac{1}{2}) + \frac{a_3 \delta}{\gamma} \\
		T & 
		T & 
		a_3 T -1 & 
		a_2 T - \frac{2\delta}{\gamma} \\
		-(1+i) \beta & 
		-(1+i) \beta & 
		-a_3(1+i) \beta & 
		1 - a_2(1+i) \beta
		\end{bmatrix}. 
	\end{equation}
	Here, 
	$\alpha = a_2^2 / (4(a_0a_2-1))$,
	$\beta = a_3^2 / (4(a_0a_3-1))$,  
	$\gamma = \beta^2(1-a_0a_2) + 2 \alpha \beta (1 - a_0a_3) + a_0 \beta - a_3(1 - A)/2$, 
	$\delta = \alpha^2 (1-a_0a_3) + 2 \alpha \beta (1 - a_0a_2) + a_0 \alpha - a_2(1 - A)/2$, 
	$A = a_3 \alpha + a_2 \beta$ in Theorem \ref{4A-octa}. 
	And 
	$S = - ( (1-i)A + i)/2$, 
	$T = \alpha (1 - i) + 2 \beta \delta/\gamma$. 
\end{thm}

\begin{proof}
	Let $p\neq 2$. 
	We start from the $4A$-octanomial normal form (\ref{eq:4A-octa}). 
	By coordinate change, 
	it is defined by the equation (\ref{eq:4A_3}): 
	\[ -x_0^2x_1 
		+ x_1^2( 
			\gamma x_2
			+ \delta x_3 )
		+(1-a_0a_3)x_2^2x_3 + (1-a_0a_2)x_2x_3^2 = 0, \]
	which admits the automorphism $\diag (i,-1,1,1)$ of type $4A$. 

	By applying the coordinate change 	
	$x_2 \mapsto (x_2 - \delta x_3)/ \gamma$, 
	the surface is defined by 
	\begin{equation}\label{eq:8A_1}
		-x_0^2 x_1 + x_1^2 x_2 
		+  \frac{1-a_0a_3}{\gamma^2} x_2^2x_3 
		+ \frac{\delta(\delta(1-a_0a_3) - \gamma(1-a_0a_2))}{\gamma^2} x_3^3 = 0.
	\end{equation}
	The term $x_2x_3^2$ vanishes, 
	and this surface admits the automorphism 
	$\diag (\zeta_8, -i, -1, 1)$ of type $8A$. 
\end{proof}

\subsection{The stratum $8A = 12A$ in $p=3$} \label{sec:8A=12A_char3}

In $p=3$, the strata $8A$ and $12A$ coincide \cite[Lemma 12.13]{DD19}. 
We will give the octanomial form of the corresponding surface which preserves the specialization 
from the strata $3A$ and $4A$.  

\begin{thm}[The $8A=12A$-octanomial normal form in $p=3$]\label{8A=12A-octa_char3}
	In $p=3$, 
	every cubic surface admitting an automorphism of type $8A$ or $12A$ is projectively isomorphic to the surface 
	given by the octanomial form with $a_0 = a_1 = 0$, $a_2 = i$, and $a_3 = -i$: 
	\begin{equation}\label{eq:8A=12A-octa_char3}
		x_0x_1 ( x_0 + x_1 -ix_2 + ix_3) + x_2x_3( x_2 + x_3) = 0.
	\end{equation}
	The surface has the automorphism of both type $8A$ and $12A$: 
	\begin{equation}\label{eq:8A=12A-auto_char3}
	\begin{bmatrix}
		-\zeta_8 + i & \zeta_8 + i & 1+i & -(1+i) \\
		\zeta_8 + i & -\zeta_8 + i & 1+i & -(1+i) \\
		-(1+i) & -(1+i) & 1+i & -i \\
		-(1+i) & -(1+i) & -1+i & -(1+i)
	\end{bmatrix}, \ 
	\begin{bmatrix}
		1 - i & 1 + i & i & -i \\
		1 + i & 1 - i & i & -i \\
		1 & 1 & -(1 + i) & -1 + i \\
		1 & 1 & 1 - i & i 
	\end{bmatrix}. 
	\end{equation}
\end{thm}

\begin{proof}
	The octanomial form (\ref{eq:8A=12A-octa_char3}) satisfies 
	the condition of the strata $3A$ (\ref{eq:3A-octa}) 
	and 4A (\ref{eq:4A-octa}), 
	where $q = 1$ and $p=0$ in Theorem \ref{4A-octa} $(\ast)$. 
	In addition, 
	since $\gamma = -i$ and $\delta = i$, 
	the condition $(\dag)$ of Theorem \ref{8A-octa} is also satisfied. 
	The automorphism of type $3A$ is given by (\ref{eq:3A-auto}), 
	which is commuting with the one of type $4A$
	\[
	\begin{bmatrix}
		1 - i & 1 + i & i & -i \\
		1 + i & 1 - i & i & -i \\
		1 & 1 & 1 - i & i  \\
		1 & 1 & -i  & 1 + i 
	\end{bmatrix}
	\]
	given by (\ref{eq:4A-auto}). 
	The product of them is the one of type $12A$. 
\end{proof}

\vspace{10pt}
\subsection*{Acknowledgments}\vspace{-4pt}
The author thanks her adviser Hisanori Ohashi for continued support, 
and thanks Ryoya Ando, Makoto Enokizono, Tomoya Namba, Kengo Maehara, Yuya Matsumoto for helpful comments and discussions. 


\end{document}